\documentclass[11pt]{amsart} 
\usepackage{amsthm,amsbsy,amsmath,amssymb,amscd,amsfonts,array,mathrsfs,verbatim,enumerate,enumitem,xypic,color}
\usepackage[all,cmtip]{xy}
\xyoption{arc} 

\newtheorem{thm}{Theorem}[section]
\newtheorem{prop}[thm]{Proposition}     
\newtheorem{lem}[thm]{Lemma}
\newtheorem{cor}[thm]{Corollary}

\theoremstyle{definition}

\newtheorem{defn}[thm]{Definition}

\newtheorem{example}[thm]{Example} 
\newtheorem{rem}[thm]{Remark}
\newtheorem{construction}[thm]{Construction}

\numberwithin{equation}{section}

\DeclareFontFamily{OT1}{rsfs}{}
\DeclareFontShape{OT1}{rsfs}{n}{it}{<-> rsfs10}{}
\DeclareMathAlphabet{\curly}{OT1}{rsfs}{n}{it}

\newcommand{\D}{{\bf D}} 

\newcommand{\NN}{\mathbb{N}} 
\newcommand{\RR}{\mathbb{R}} 
\newcommand{\PP}{\mathbb{P}} 
\newcommand{\CC}{\mathbb{C}} 
\newcommand{\GG}{\mathbb{G}} 
\newcommand{\FF}{\mathbb{F}} 

\newcommand{\MM}{V} 

\newcommand{\EE}{\mathbb{E}}
\newcommand{\DD}{\mathbb{D}}
\newcommand{\SG}{\mathfrak{S}}

\newcommand{\so}{\mathfrak{so}}

\newcommand{\ov}{\overline}
\newcommand{\into}{\hookrightarrow}

\newcommand{\be}{\begin{eqnarray*}}
\newcommand{\ee}{\end{eqnarray*}}
\newcommand{\bne}[1]{\begin{eqnarray} \label{#1} }
\newcommand{\ene}{\end{eqnarray}}

\newcommand{\bp}{
   \arraycolsep=6pt 
   \def\arraystretch{1}
   \begin{pmatrix}  
}
\newcommand{\ep}{
   \end{pmatrix}
   \arraycolsep=2pt 
   \def\arraystretch{1.2}
}

\newcommand{\xym}{\xymatrix}

\newcommand{\slot}{ \hspace{0.05in} {\rm \_} \hspace{0.05in} } 

\newcommand{\Hom}{\operatorname{Hom}}   

\renewcommand{\H}{\operatorname{H}}
\newcommand{\R}{\operatorname{R}} 

\newcommand{\Ker}{\operatorname{Ker}}
\newcommand{\Cok}{\operatorname{Cok}}

\newcommand{\Id}{\operatorname{Id}}

\newcommand{\Tr}{\operatorname{Tr}}    
\newcommand{\Diag}{\operatorname{Diag}}

\newcommand{\sdeg}{\operatorname{sdeg}}

\def\arraystretch{1.2} 

\begin{document}


\author{H.~I.~Bozma}
\address{Department of Electrical and Electronics Engineering, Bo\u{g}azi\c{c}i University}
\email{bozma@boun.edu.tr}

\author{W.~D.~Gillam}
\address{Robbinston, Maine}
\email{wdgillam@gmail.com}

\author{F.~\"Ozt\"urk}
\address{Department of Mathematics, Bo\u{g}azi\c{c}i University}
\email{ferit.ozturk@boun.edu.tr}

\date{\today}
\title[Morse-Bott functions on orthogonal groups]{Morse-Bott functions on orthogonal groups}

\begin{abstract}  We make a detailed study of various (quadratic and linear) Morse-Bott trace functions on the orthogonal groups $O(n)$.  
We describe the critical loci of the quadratic trace function $\Tr(AXBX^T)$ and determine their indices via perfect fillings of tables associated with the multiplicities of the eigenvalues of $A$ and $B$.
We give a simplified treatment of T.~Frankel's analysis of the linear trace function on $SO(n)$, as well as a combinatorial explanation of the relationship between the mod $2$ Betti numbers of $SO(n)$ and those of the Grassmannians $\GG(2k,n)$ obtained from this analysis.  We review the basic notions of Morse-Bott cohomology in a simple  case where the set of critical points has two connected components.  We then use these results to give a new Morse-theoretic computation of the mod $2$ Betti numbers of $SO(n)$. 
\end{abstract}

\maketitle

\section{Introduction} \label{section:intro}

This paper is devoted to a strictly Morse-theoretic study of various functions on the orthogonal groups $O(n)$.  Many of our results surely generalize to the classical groups $U(n)$ and $Sp(n)$ by replacing the base field $\RR$ with $\CC$ or the quaternion skew-field and making the usual modifications (replace transposes with conjugate-transposes and traces with their real parts).  For brevity and simplicity, however, we work strictly with the usual orthogonal groups and leave any generalizations to the interested reader.

We are essentially interested in  two classes of functions. The first is the quadratic trace function $f_{A,B} : O(n) \to \RR$ given as $f_{A,B}(X)=\Tr(AXBX^T)$, for some fixed, orthogonally diagonalizable $n \times n$ matrices $A,B$.  
In the case $A$ and $B$ are symmetric matrices, the question of finding the extrema of $f_{A,B}$  restricted to the signed permutation matrices  is nothing but the well-known Quadratic Assignment Problem.
The extrema problem of $f_{A,B}$ was worked out by von Neumann back in 1937 \cite{vN}. 
In a rather general setting, the fact that  $f_{A,B}$ is a Morse-Bott function was proven in Lie theoretical terms  in  \cite{DKV}. 
In Section~\ref{section:BMF1}  of the present article we reprove that $f_{A,B}$ is Morse-Bott.
The content is simply something of a tour-de-force of linear algebra and is self-contained. The benefit of this 
presentation is a complete description
of the critical loci of $f_{A,B}$ and their indices.  We show that the critical loci are quotients of products of orthogonal groups, the topology of which is explicitly determined by the combinatorial objects called the perfect fillings of tables with margins prescribed by the multiplicities of the eigenvalues of $A$ and $B$.
Moreover the  index of each connected component of the critical locus can be computed via the corresponding perfect filling.  In the special case when $A$ and $B$ are matrices with distinct eigenvalues, the indices are nothing but  the ``inversion numbers" of permutations. 

Second, we study functions $O(n) \to \RR$ obtained by  restricting a linear function on the vector space of all $n \times n$ matrices.  These functions were also studied in \cite{SS}, where the authors determine which of these functions are Morse.  There are two such functions to which we devote special attention (neither of which is \emph{Morse} for general $n$).  One such is the function $f(X) = \Tr(X)$, originally studied by T.~Frankel in \cite{F}.  In Section~\ref{section:SOn1} we give a self-contained derivation of Frankel's results which seems simpler to us than Frankel's original approach and which is more in the spirit of the rest of our paper.  The gist of these results is that $f$ is Morse-Bott and the critical locus of $f$ is a disjoint union of Grassmannians.  In the Appendix we also give a purely combinatorial discussion of some related results of Frankel.

The other linear function of particular interest is the function $f_{nn} : SO(n) \to \RR$, $f_{nn}(X) = X_{nn}$ obtained by taking the lower right entry of $X$.  (Any entry would do, but this is a convenient choice.)  We show very easily in Section~\ref{section:SOn2} that $f_{nn}$ is Morse-Bott and that the critical locus of $f_{nn}$ is a disjoint union of two copies of $SO(n-1)$.  The methods of ``Morse-Bott cohomology" (which we treat independently in Section~\ref{section:BMC} in an original manner catering to our situation) yield a long exact sequence relating the cohomology of $SO(n)$ to that of $SO(n-1)$.  The novelty of the Morse-theoretic point of view we take is to interpret the connecting maps in this long exact sequence in Morse-theoretic terms.  This allows us to show that these maps are zero with $\FF_2$-coefficients.  We thus obtain a recursive description of the mod $2$ Betti numbers of $SO(n)$ which is easily solved to yield a simple combinatorial formula for these numbers.  Although these numbers can be computed in a variety of ways, we believe our Morse-theoretic computation is quite simple and  natural.  (Together with our combinatorial results in the Appendix, Frankel's study also yields the same Betti number formulae, though that approach is considerably more complicated as it relies on knowing the Betti numbers of Grassmannians, as well as on non-Morse-theoretic results of E.~E.~Floyd.)

\noindent {\bf Acknowledgements.} The work of the last author is partially supported by Bo\u{g}azi\c{c}i University Research Project BAP-17B06P2.

\section{Some Morse-Bott functions on orthogonal groups} \label{section:BMF1}

Fix two symmetric (equivalently, orthogonally diagonalizable) $n \times n$ matrices $A$ and $B$ and consider the smooth function \be f=f_{A,B}:O(n) & \to & \RR \\ \nonumber f(X) & := & \Tr(AXBX^T).\ee  If $A' = Q^TAQ$, $B'=R^TBR$ are orthogonal conjugates of $A$ and $B$, then, using conjugation-invariance of the trace, one sees that $f_{A',B'}(X) = f_{A,B}(QXR^T)$, hence $f_{A',B'}$ is just the composition of $f_{A,B}$ and the automorphism $X \mapsto QXR^T$ of $O(n)$.  Therefore, with no real loss of generality, we will assume for the remainder of this section that \be A & = & \Diag(a_1I_{m_1},\dots,a_s I_{m_s}) \\ \label{BFB} B & = & \Diag(b_1I_{n_1},\dots,b_t I_{n_t}), \ee with $a_i \neq a_k$ for $i \neq k$ and $b_j \neq b_l$ for $j \neq l$.  The eigenvalue multiplicities $m_i$, $n_j$ satisfy \be & \sum_{i=1}^s m_i = \sum_{j=1}^t n_j = n. \ee
Throughout this section, $i$ and $k$ (resp.\ $j$ and $l$) will always denote elements of the set $\{ 1, \dots, s \}$ (resp.\ $\{ 1, \dots, t \}$).  We write $F := \{ X \in O(n) : (Df)(X)=0 \}$ for the critical locus of $f$.

\begin{lem} \label{lem:1} For a symmetric $n \times n$ matrix $S$ and a diagonal $n \times n$ matrix $A$ as above, the following are equivalent: \begin{enumerate}[label=(\roman*), ref=\roman*] \item $AS$ is symmetric. \item $AS=SA$. \item $S=\Diag(S_1,\dots,S_s)$ with each $S_i$ a symmetric $m_i \times m_i$ matrix. \end{enumerate} \end{lem}


The proof is left as an exercise. The metric \be \langle M,N \rangle & := & \Tr(M^TN) \ee on $T_X O(n) = \{ M : MX^T+XM^T = 0 \}$ 
is the unique (up to scaling) bi-invariant Riemannian metric on O(n).

\begin{lem} \label{lem:nabla}  The gradient of $f=f_{A,B}$ at a point $X \in O(n)$ is given by \bne{nablaformula} (\nabla f)(X) &=& (AXBX^T-XBX^TA)X. \ene  The following are equivalent: \begin{enumerate}[label=(\roman*), ref=\roman*] \item $X \in F$ (i.e.\ $X$ is a critical point of $f$). \item $AXBX^T$ is symmetric. \item $AXBX^T=XBX^TA$ \item $XBX^T=\Diag(H_1,\dots,H_s)$ with each $H_i$ a symmetric $m_i \times m_i$ matrix. \end{enumerate} \end{lem}

\begin{proof}  In fact, more generally, for \emph{any} $n \times n$ matrices $A$ and $B$, the gradient of the function $f : O(n) \to \RR$ defined by $f(X) := \Tr(AXBX^T)$ will be given by $(\nabla f)(X) = (A^T X B^T X^T - XB^T X^T A^T)X$.  To see this, one first computes that the derivative of $f$ at $X$ is given by \bne{Df} (Df)(X)(M) & = & \Tr(AMBX^T+AXBM^T) \ene for $X \in O(n)$, $M \in T_X O(n)$.  We then compute \be \langle (\nabla f)(X),M \rangle &=& \Tr( (\nabla f)^T M) \\ & = & \Tr(BX^TAM-X^TAXBX^TM) \\ & = & \Tr(AMBX^T -AXBX^TMX^T) \\ & = & \Tr(AMBX^T + AXBM^T) \\ & = & (Df)(X)(M) \ee using $MX^T+XM^T=0$ (since $M \in T_X O(n)$) and standard properties of the trace.  The equivalence of the first two conditions is evident from formula \eqref{nablaformula}.  For the equivalence with the other conditions, apply Lemma~\ref{lem:1} with $S=XBX^T$. \end{proof}

\begin{defn} \label{defn:perfectfilling} A \emph{perfect filling} with margins 
$(m_1,\ldots,m_t;n_1,\ldots,n_s)$ is an $s \times t$ matrix $\epsilon$ with entries $\epsilon_{ij}$ in $\NN = \{ 0,1,\dots \}$ satisfying \bne{3} & \sum_{j=1}^t \epsilon_{ij} = m_i {\rm \; for \; each \;} i \in \{ 1, \dots , s \} \\ \label{2} & \sum_{i=1}^s \epsilon_{ij} = n_j {\rm \; for \; each \;} j \in \{ 1, \dots , t \}   \ene \end{defn}

\begin{example} \label{example:distincteigenvalues} If $A$ and $B$ have distinct eigenvalues, then all $m_i$ and $n_j$ are $1$, $s=t=n$, and a perfect filling is an $n \times n$ permutation matrix. \end{example}

Throughout, we set $$ O(\ov{m}) :=  \prod_{i=1}^s O(m_i), \quad O(\ov{n})  :=  \prod_{j=1}^t O(n_j), \quad O(\epsilon)  :=  \prod_{i,j} O(\epsilon_{ij}).$$

\begin{construction} \label{const:main}  Given $Q = (Q[1],\dots,Q[s]) \in O(\ov{m})$, $R = (R[1],\dots,R[t]) \in O(\ov{n})$, and a perfect filling $\epsilon$, we construct an $n \times n$ matrix $X=\Phi_{\epsilon}(Q,R)$ as follows:  We write the $m_i \times m_i$ matrix $Q[i]$ in block form \bne{BFQ} Q[i] & = & \bp Q[i,1] & \cdots & Q[i,t] \ep, \ene with $Q[i,j]$ of size $m_i \times \epsilon_{ij}$.  (This makes sense in light of \eqref{3}.)  Similarly, we write the $n_j \times n_j$ matrix $R[j]$ in block form \bne{BFR} R[j] & = & \bp R[1,j] & \cdots & R[s,j] \ep, \ene with $R[i,j]$ of size $n_j \times \epsilon_{ij}$.  (This makes sense in light of \eqref{2}.)  We let $X[i,j]$ be the $m_i \times n_j$ matrix defined by \bne{XDef} X[i,j] & := & Q[i,j]R[i,j]^T \ene and we define $X$ to be the $n \times n$ matrix written in block form as \bne{BFX} X & = & \bp X[1,1] & \cdots & X[1,t] \\ \vdots & & \vdots \\ X[s,1] & \cdots & X[s,t] \ep. \ene

The Lie group $O(\epsilon)$ acts (smoothly, on the right) on $O(\ov{m}) \times O(\ov{n})$ by setting \be (Q,R) \cdot U & := & (Q \cdot U, R \cdot U) \\ \nonumber Q \cdot U & := & (Q[1] \cdot U, \dots, Q[s] \cdot U) \\ \nonumber R \cdot U & := & (R[1] \cdot U, \dots, R[t] \cdot U) \\ \nonumber Q[i] \cdot U & := & \bp Q[i,1] U[i,1] & \cdots & Q[i,t] U[i,t] \ep \\ \nonumber R[j] \cdot U & := & \bp R[1,j] U[1,j] & \cdots & R[s,j] U[s,j] \ep \ee for $U=(U[i,j]) \in O(\epsilon)$.  In other words: \bne{UQR2} (Q \cdot U)[i,j] & = & Q[i,j]U[i,j] \\ \nonumber (R \cdot U)[i,j] & = & R[i,j] U[i,j]. \ene From \eqref{UQR2}, \eqref{XDef}, and \eqref{BFX}, we find \be \Phi_{\epsilon}(Q,R) & = & \Phi_{\epsilon}(Q \cdot U, R \cdot U). \ee \end{construction}

\begin{rem} \label{rem:Oepsilonaction} Since the columns of $Q[i,j]$ are linearly independent (even orthonormal), the action of $U[i,j]$ on $Q[i,j]$ is free.  The action of $O(\epsilon)$ on $O(\ov{m}) \times O(\ov{n})$ is therefore a free, smooth action of a compact Lie group on a smooth, compact manifold.  The quotient $(O(\ov{m}) \times O(\ov{n}))/O(\epsilon)$ therefore admits a unique smooth manifold structure for which the quotient map is submersive.  The quotient is understood to have this smooth structure throughout. \end{rem}

\begin{prop} \label{prop:X} For $Q$, $R$, $\epsilon$, and $X=\Phi_{\epsilon}(Q,R)$ as in Construction~\ref{const:main}: \begin{enumerate}[label=(\roman*), ref=\roman*] \item \label{X1} $X \in O(n)$ and $X$ is a critical point of $f : O(n) \to \RR$. \item \label{X2} $XBX^T = \Diag(H_1,\dots,H_s)$, where $H_i$ is the $m_i \times m_i$ symmetric matrix given by \be H_i &=& \sum_{j=1}^t b_j Q[i,j]Q[i,j]^T. \ee \item \label{X3} The columns of $Q[i,j]$ form an orthonormal basis for the $b_j$-eigenspace of $H_i$.  In particular, the dimension of this eigenspace is $\epsilon_{ij}$. \end{enumerate} \end{prop}

\begin{proof}  The proof is basically an exercise in multiplying matrices in block form.

The matrix $Q[i] \in O(m_i)$ satisfies $Q[i]Q[i]^T = I_{m_i}$ and $Q[i]^TQ[i] = I_{m_i}$.  Writing these products in terms of the block form \eqref{BFQ}, we find \bne{Q1} \nonumber \sum_{j=1}^t Q[i,j]Q[i,j]^T & = & I_{m_i} \\ \label{Q2} Q[i,j]^TQ[i,l] & = & \left \{ \begin{array}{lll} I_{\epsilon_{ij}}, & \quad & j=l \\ 0_{\epsilon_{ij} \times \epsilon_{il}}, & & j \neq l. \end{array} \right . \ene Similarly writing $R[j] \in O(n_j)$ in the form \eqref{BFR}, we find \bne{R1} \sum_{i=1}^s R[i,j]R[i,j]^T & = & I_{n_j} \\ \label{R2} R[i,j]^TR[k,j] & = & \left \{ \begin{array}{lll} I_{\epsilon_{ij}}, & \quad & i=k \\ 0_{\epsilon_{ij} \times \epsilon_{kj}}, & & i \neq k. \end{array} \right . \ene

To see that $X \in O(n)$, we compute \be X^TX & = & \bp X[1,1]^T & \cdots & X[s,1]^T \\ \vdots & & \vdots \\ X[1,t]^T & \cdots & X[s,t]^T \ep \bp X[1,1] & \dots & X[1,t] \\ \vdots & & \vdots \\ X[s,1] & \cdots & X[s,t] \ep \\ & = & \bp M[1,1] & \cdots & M[1,t] \\ \vdots & & \vdots \\ M[t,1] & \cdots & M[t,t] \ep, \ee where $M[j,l]$ is the $n_j \times n_l$ matrix given by \be M[j,l] & = & \sum_{i=1}^s X[i,j]^T X[i,l]. \ee  Using the definition \eqref{XDef} of the $X[i,j]$ and \eqref{Q2} and \eqref{R1} above, we compute \be M[j,l] & = & \sum_{i=1}^s R[i,j] Q[i,j]^T Q[i,l] R[i,l]^T \\ & = & \left \{ \begin{array}{lll} \sum_{i=1}^s R[i,j]R[i,j]^T, & \quad & j=l \\ 0_{n_j \times n_l}, & & j \neq l \end{array} \right . \\ & = & \left \{ \begin{array}{lll} I_{n_j}, & \quad & j=l \\ 0_{n_j \times n_l}, & & j \neq l. \end{array} \right . \ee  This proves that $X \in O(n)$.

\emph{Observe:}  If it were known that the $Q[i]$ are orthonormal (so \eqref{Q2} holds) and $X$ is orthonormal, then the above computation would show that the $R[j]$ satisfy \eqref{R1} and are hence orthonormal.

For \eqref{X2}, we compute \be XBX^T & = & \bp X[1,1] & \cdots & X[1,t] \\ \vdots & & \vdots \\ X[s,1] & \cdots & X[s,t] \ep \bp b_1 X[1,1]^T & \cdots & b_1 X[s,1]^T \\ \vdots & & \vdots \\ b_t X[1,t]^T & \cdots & b_t X[s,t]^T \ep \\ & = & \bp D[1,1] & \cdots & D[1,s] \\ \vdots & & \vdots \\ D[s,1] & \cdots & D[s,s] \ep, \ee where $D[i,k]$ is the $m_i \times m_k$ matrix given by \be D[i,k] & = & \sum_{j=1}^t b_j X[i,j]X[k,j]^T . \ee  Expanding this out using the definition \eqref{XDef} of the $X[i,j]$ and \eqref{Q2}, we find \be D[i,k] & = & \left \{ \begin{array}{lll} \sum_{j=1}^t b_j Q[i,j]Q[i,j]^T , & \quad & i=k \\ 0_{m_i \times m_k}, & & i \neq k. \end{array} \right .  \ee  This proves \eqref{X2}.  We have \eqref{X2} with $X\in O(n) \implies$\eqref{X1} by Lemma~\ref{lem:nabla}.

For \eqref{X3}, we compute, using \eqref{X2} and \eqref{Q2}: \be H_i Q[i,j] & = & \sum_{l=1}^t b_l Q[i,l] Q[i,l]^T Q[i,j] \\ & = & b_j Q[i,j]. \ee  This shows that the $\epsilon_{ij}$ orthonormal columns of $Q[i,j]$ are all in the $b_j$ eigenspace of $H_i$, so this eigenspace has dimension $\geq \epsilon_{ij}$.  Since this is true for each $j$, $H_i$ is an $m_i \times m_i$ matrix, and \eqref{3} holds, this last inequality must actually be an equality for every $j$ by basic linear algebra.

\end{proof}

\begin{thm} \label{thm:main} Construction~\ref{const:main} yields a diffeomorphism \be \Phi = \coprod_{\epsilon} \Phi_\epsilon : \coprod_\epsilon ( O(\ov{m}) \times O(\ov{n}))/O(\epsilon) & \to & F \ee onto the critical locus $F$ of $f$ (the coproduct is over perfect fillings $\epsilon$). 
\end{thm}

\begin{proof}  Fix $X \in F$.  Set $H := XBX^T$.  Since $X \in F$, \bne{BFH} H & = & \Diag(H_1,\dots,H_s), \ene with $H_i$ an $m_i \times m_i$ symmetric matrix (Lemma~\ref{lem:nabla}).  Let $E_{ij} \subseteq \RR^{m_i}$ be the $b_j$-eigenspace of $H_i$ and $\epsilon_{ij}$ be its dimension.  

\noindent {\bf Claim 1:} $\epsilon=(\epsilon_{ij})$ is a perfect filling.

Fix any $i$.  From the block form \eqref{BFH} of $H$ we see that any eigenvalue of $H_i$ must also be an eigenvalue of $H$.  Since $H$ is similar to $B$ the eigenvalues of $H$ are the $b_j$.  Therefore the eigenvalues of $H_i$ are \emph{among} the $b_j$.  Since the symmetric $m_i \times m_i$ matrix $H_i$ is diagonalizable, the equality $\sum_j \epsilon_{ij} = m_i$ follows.  Now fix any $j$.  From the block form \eqref{BFH} of $H$ we see that the $b_j$-eigenspace $E_j$ of $H$ is the direct sum, over $i$, of the $E_{ij}$.  Since $H$ is similar to $B$ we have $\dim E_j = n_j$.  The equality $\sum_i \epsilon_{ij} = n_j$ follows.  This proves the claim.

Now \emph{choose}, for each $i,j$, an $m_i \times \epsilon_{ij}$ matrix $Q[i,j]$ whose columns form an orthonormal basis for $E_{ij} \subseteq \RR^{m_i}$.  By definition of $H$ (and the fact that $X \in O(n)$), we have \bne{HXXB} HX & = & XB. \ene  Writing $H$, $X$, and $B$ in the (respective) block forms \eqref{BFH}, \eqref{BFX}, \eqref{BFB} and expanding out \eqref{HXXB}, we find \be H_i X[i,j] & = & b_j X[i,j] \ee for all $i,j$.  Therefore each column of $X[i,j]$ is in $E_{ij}$.  Since the columns of $Q[i,j]$ form a basis for $E_{ij}$ there is a \emph{unique} $n_j \times \epsilon_{ij}$ matrix $R[i,j]$ such that \bne{usualway} X[i,j] &=& Q[i,j]R[i,j]^T.\ene  Define $Q[i]$ and $R[j]$ from the $Q[i,j]$ and $R[i,j]$ using the usual block forms \eqref{BFQ}, \eqref{BFR}.

\noindent {\bf Claim 2:} $Q[i] \in O(m_i)$ for each $i$ and $R[j] \in O(n_j)$ for each $j$.

The $Q[i,j]$ have orthonormal columns forming a basis for $E_{ij}$, so to show that $Q[i] \in O(m_i)$, it is enough to prove that $E_{ij} \perp E_{il}$ for $j \neq l$.  This holds by basic linear algebra since $E_{ij}$ and $E_{il}$ are distinct eigenspaces of the \emph{symmetric} matrix $H_i$.  Since $Q[i] \in O(m_i)$ for each $i$, $X \in O(n)$, and we have the relationship \eqref{usualway}, we have $R[j] \in O(n_j)$ by the \emph{observation} made in the proof of Proposition~\ref{prop:X}\eqref{X1}.  This proves Claim 2.

By Claims 1 and 2, $Q$, $R$, and $\epsilon$ constructed above are as in Construction~\ref{const:main}; clearly we have $X=\Phi_{\epsilon}(Q,R)$ (compare \eqref{usualway} and \eqref{XDef}).  Suppose we also have $X=\Phi_{\epsilon'}(Q',R')$ with $Q'$, $R'$, $\epsilon'$ as in Construction~\ref{const:main}.  If we define the $m_i \times m_i$ matrices $H_i$ from $X$ and $B$ as at the beginning of this proof, then by Proposition~\ref{prop:X}\eqref{X3}, the columns of $Q[i,j]$ form an orthonormal basis for the $b_j$-eigenspace $E_{ij} \subseteq \RR^{m_i}$ of $H_i$, as do the columns of $Q'[i,j]$.  In particular, $Q[i,j]$ and $Q'[i,j]$ must have the same number of columns, so we must have $\epsilon_{ij} = \epsilon'_{ij}$.  Furthermore, there is a unique $U[i,j] \in O(\epsilon_{ij})$ with $Q'[i,j]=Q[i,j]U[i,j]$.  Combining this with the fact that $$X[i,j] = Q[i,j]R[i,j]^T = Q'[i,j]R'[i,j]^T$$ (formula \eqref{XDef} in Construction~\ref{const:main}), we deduce \bne{cancel} Q'[i,j] U[i,j]^T R[i,j]^T &=& Q'[i,j] R'[i,j]^T. \ene  By \eqref{Q2} (with $Q$ replaced by $Q'$), we can multiply \eqref{cancel} on the left by $Q'[i,j]^T$ to cancel the $Q'[i,j]$'s on the left of each side of \eqref{cancel}, then transpose to find $R'[i,j] = R[i,j]U[i,j]$.  This proves $(Q',R') = (Q,R) \cdot U$, where $U=(U[i,j]) \in O(\epsilon)$.

The above results (and Proposition~\ref{prop:X}\eqref{X1}) demonstrate that $\Phi$ (which is clearly smooth) is bijective.  It must therefore be a homeomorphism since its domain is compact and its codomain is Hausdorff.  The domain of $\Phi$ is a smooth manifold (Remark~\ref{rem:Oepsilonaction}), so to conclude that it is a diffeomorphism onto its image $F$, it remains only to prove that the derivative of each $\Phi_{\epsilon}$ is injective when we view $\Phi_{\epsilon}$ as a map $(O(\ov{m}) \times O(\ov{n})) / O(\epsilon) \to O(n)$.  Equivalently, if we view $\Phi_{\epsilon}$ simply as a map $\Phi_{\epsilon} : O(\ov{m}) \times O(\ov{n}) \to O(n)$, then we must prove that the kernel of each derivative $(D \Phi_{\epsilon})(Q,R)$ is precisely the tangent space (at $(Q,R)$) to the $O(\epsilon)$-orbit of $(Q,R)$.  The latter is the image of the Lie derivative (derivative at the identity) of the orbit map $U \mapsto (Q,R) \cdot U$ from $O(\epsilon)$ to $O(\ov{m}) \times O(\ov{n})$.  

Now we compute these derivatives.  For $(M,N) \in T_Q O(\ov{m}) \oplus T_R O(\ov{n})$, we find \bne{DPhi} (D \Phi_{\epsilon})(Q,R)(M,N) & = & \bp D[1,1] & \cdots & D[1,t] \\ \vdots & & \vdots \\ D[s,1] & \cdots & D[s,t] \ep, \ene where \bne{Dij} D[i,j] & = & M[i,j] R[i,j]^T + Q[i,j]N[i,j]^T. \ene  (We have broken the components $M[i]$ of $M$ into blocks $M[i,j]$ of size $m_i \times \epsilon_{ij}$ in the same way we broke up the $Q[i]$ in \eqref{BFQ}.  Similarly, we have broken the $N[j]$ into blocks $N[i,j]$ of size $n_j \times \epsilon_{ij}$, just as we broke up the $R[i]$ in \eqref{BFR}.)  The aforementioned Lie derivative takes $$P=(P[i,j]) \in T_I O(\epsilon) = \bigoplus_{i,j} \so(\epsilon_{ij}) $$ to $(M,N)$ where, when cut into blocks in the usual manner, \bne{DLie} M[i,j] & = & Q[i,j]P[i,j] \\ \label{DLie2} N[i,j] & = & R[i,j]P[i,j]. \ene  (Notice that, from \eqref{Dij} and skew-symmetry of the $P[i,j]$, such an $(M,N)$ will be in the kernel of $(D \Phi_{\epsilon})(Q,R)$.  This is, of course, just an ``infinitesimal" version of the known $O(\epsilon)$-invariance of $\Phi_{\epsilon}$.)  Given any $(M,N) \in T_Q O(\ov{m}) \oplus T_R O(\ov{n})$ (not necessarily in the kernel of $(D \Phi_{\epsilon})(Q,R)$), we can define $\epsilon_{ij} \times \epsilon_{ij}$ matrices $P[i,j]$ by \bne{Pij} P[i,j] & := & Q[i,j]^T M[i,j]. \ene

\noindent {\bf Claim 3:} The matrices $P[i,j]$ in \eqref{Pij} are skew-symmetric.

To see this, note that \bne{MQT} M[i]Q[i]^T + Q[i]M[i]^T & = & 0 \ene since $M[i] \in T_{Q[i]} O(m_i)$.  Writing \eqref{MQT} in terms of the usual ``block forms," we see that \eqref{MQT} is equivalent to \bne{BFMQT} \sum_j M[i,j]Q[i,j]^T + Q[i,j]M[i,j]^T & = & 0. \ene  Multiplying \eqref{BFMQT} on the right by $Q[i,l]$ and on the left by $Q[i,l']^T$ and using \eqref{Q2}, we find \be Q[i,l']^T M[i,l] + M[i,l']^T Q[i,l] & = & 0, \ee  Setting $l=j$ and $l'=j$  here proves Claim 3.  For later use, let us make an analogous computation.  The condition \be  N[j]R[j]^T + R[j]N[j]^T & = & 0 \ee for $N[j]$ to be in $T_{R[j]} O(n_j)$ is equivalent to \be \sum_i N[i,j]R[i,j]^T + R[i,j]N[i,j]^T & = & 0. \ee  Multiplying on the right by $R[k,j]$ and on the left by $R[k',j]^T$ and using \eqref{R2} yields \bne{seven} R[k',j]^T N[k,j]+N[k',j]^T R[k,j] & = & 0. \ene

Now, if $(M,N)$ is actually in the kernel of $(D \Phi_{\epsilon})(Q,R)$ (i.e.\ all the $D[i,j]$ in \eqref{Dij} are zero), then of course we claim that $(M,N)$ is the Lie derivative evaluated at the $P[i,j]$ defined in \eqref{Pij}.  To see this, we need to establish \eqref{DLie} and \eqref{DLie2} when the $P[i,j]$ are defined by \eqref{Pij}.  For \eqref{DLie}, we just multiply \eqref{Pij} on the left by $Q[i,j]$ and use \eqref{Q2}.  (We don't need to know $(D \Phi_\epsilon)(Q,R)(M,N)=0$ yet.)  For \eqref{DLie2}, we multiply $D[i,j]=0$ (using formula \eqref{Dij} for $D[i,j]$) on the left by $Q[i,j]^T$ and use \eqref{Q2} and \eqref{Pij} to find \be P[i,j] R[i,j]^T + N[i,j]^T &=& 0. \ee  By Claim 3 we can substitute $-P[i,j]^T$ for $P[i,j]$ here to obtain (the transpose of) \eqref{DLie2}.  
\end{proof}

\begin{rem} 
Though established by our results, it is perhaps not obvious from the formulae \eqref{DPhi} and \eqref{Dij} that $(D \Phi_{\epsilon})(Q,R)(M,N)$ is actually in $T_X O(n)$ ($X = \Phi_\epsilon(Q,R)$).  (The fact that this is true is an infinitesimal version of Proposition~\ref{prop:X}\eqref{X1}.) Since we will need some of the calculations later anyway, let us explain this point.  The condition \be VX^T+XV^T = 0 \ee for an $n \times n$ matrix $V$ to be in $T_X O(n)$, when $V$ is broken into $m_i \times n_j$ blocks $V[i,j]$ (as we broke $X$ into such blocks in \eqref{BFX}), is equivalent to the equations \be \sum_j V[i,j]X[k,j]^T + X[i,j]V[k,j]^T & = & 0 \ee for every $i,k$.  To see that this holds when $V=(D \Phi)(Q,R)(M,N)$ (so $V[i,j]$ is the $D[i,j]$ in \eqref{Dij}), we compute (using \eqref{XDef}) \bne{sumoverj} & & \sum_j D[i,j]X[k,j]^T + X[i,j]D[k,j]^T \\ \nonumber & = & \sum_j M[i,j]R[i,j]^TR[k,j]Q[k,j]^T + Q[i,j]N[i,j]^T R[k,j] Q[k,j]^T \\ \nonumber & & + Q[i,j]R[i,j]^TR[k,j]M[k,j]^T + Q[i,j]R[i,j]^T N[k,j]Q[k,j]^T . \ene  For each $j$, the second and forth terms in the sum over $j$ cancel by \eqref{seven}.  Furthermore, by \eqref{R2}, the first and third terms are also zero when $i \neq k$ and, when $i=k$, \eqref{sumoverj} is nothing but the sum known to be zero by \eqref{BFMQT}. \end{rem}

Our next task is to describe the Hessian of $f$ at a critical point $X \in F$.  We identify $T_X O(n)$ with the space $\so(n) = T_I O(n)$ of skew-symmetric matrices via the usual isomorphism \bne{TangentSpaceIso} X : \so(n) & \to & T_X O(n) = \{ M : MX^T+XM^T=0 \} \\ \nonumber E & \mapsto & XE, \ene thus we view the Hessian $H(f,X)$ of $f$ at $X$ as a quadratic form on $\so(n)$.

\begin{lem} \label{lem:Hessian} Viewing the Hessian $H=H(f,X)$ of $f$ at a critical point $X \in F$ as a quadratic form on $\so(n)$ as above, it is given by \bne{Hessianformula} H(E,N) &=& \Tr(AX[E,[N,B]]X^T) \ene for $E,N \in \so(n)$, where $[U,V] := UV-VU$.\footnote{It would be incorrect to refer to $[ \; , \; ]$ as the ``Lie bracket on $\so(n)$" because, in \eqref{Hessianformula}, we apply it to matrices that aren't in $\so(n)$.} \end{lem}

\begin{proof}  Let $Y$ be any smooth manifold, $f : Y \to \RR$ a smooth function.  View the $1$-form $df$ as a section $df : Y \to T^* Y$ of the cotangent bundle $\pi : T^* Y \to Y$.  The derivative of $df$ at a point $y \in Y$ is a linear map \bne{Ddf} (D(df))(y) : T_y T \to T_{(df)(y)} T^* Y.\ene  Now, for any vector bundle $V$ over $Y$, the restriction of $T V$ (the tangent bundle of the ``total space" of $V$) to the zero section is \emph{canonically} identified with $TY \oplus V$ by using the derivative of the zero section as a section of the derivative of the projection $\pi : V \to Y$.  (The relative tangent bundle of $\pi$ is canonically identified with $\pi^*V$.)  If $y$ is a critical point of $f$, then $(df)(y)$ lies in the zero section of $T^*Y$, hence, via the aforementioned canonical isomorphism, \eqref{Ddf} may be viewed as a map \bne{Ddf2} (D(df))(y) : T_y Y \to T_y Y \oplus T^*_y Y.\ene  The $T_y Y \to T_y Y$ component of \eqref{Ddf2} is the identity (since $df$ is a section of $\pi$); the other (``vertical") component $H : T_y Y \to T_y^* Y$ is a map from a vector space to its dual which is readily seen to be self-dual (this reflects the fact that $d^2 f = 0$).  As the notation suggests, $H$ is the Hessian of $f$ at $y$.  (A quadratic form on a vector space $V$ is the same thing as a self-dual linear map $V \to V^*$.)

In the situation at hand ($Y=O(n)$), the tangent bundle of $Y$ (hence also the cotangent bundle of $Y$) is trivial, so the situation simplifies a bit.  We trivialize $T O(n)$ by identifying $\so(n)$ and $T_X O(n)$ via the isomorphism \eqref{TangentSpaceIso}.  (This isomorphism arises as the Lie derivative of the ``left multiplication by $X$" map $X : O(n) \to O(n)$.)  With this understanding, $df : O(n) \to T^* O(n) = O(n) \times \so(n)^*$ is given by the identity in the first component (it is a section!) and by \bne{dfver} (df)^{ver} : O(n) & \to & \so(n)^* \\ \nonumber (df)^{ver}(X)(N) &=& \Tr(AXNBX^T+AXBN^TX^T) \\ \nonumber &=& \Tr(AX[N,B]X^T) \ene in the ``vertical direction".  (We are just rewriting \eqref{Df} in terms of the isomorphism \eqref{TangentSpaceIso} and using the fact that $N \in \so(n)$ is skew-symmetric.)  The derivative of $(df)^{ver}$ at $X \in O(n)$ is readily calculated to be \bne{Ddfver} (D(df)^{ver})(X) : T_X O(n) & \to & \so(n)^* \\ \nonumber (D(df)^{ver})(X)(M)(N) & = & \Tr( AMNBX^T+AXNBM^T \\ \nonumber & & + AMBN^TX^T + AXBN^T M^T ) \\ \nonumber &=& \Tr( AM[N,B]X^T+AX[N,B]M^T). \ene  (We have used the fact that $N \in \so(n)$ is skew-symmetric for the last equality.)  If we take into account the isomorphism \eqref{TangentSpaceIso}, then \eqref{Ddfver} becomes the map \bne{Hessian} H(X) : \so(n) & \to & \so(n)^* \\ \nonumber  H(X)(E,N) & = & \Tr(AXE[N,B]X^T+AX[N,B]E^TX^T). \ene  If $X$ is a critical point, then comparing this calculation with the general definition of the Hessian in the previous paragraph, we see that \eqref{Hessian} \emph{is} the Hessian of $f$ at $X$, as the notation suggests.  The proof is completed by using the fact that $E \in \so(n)$ is skew-symmetric to rewrite \eqref{Hessian} as in \eqref{Hessianformula}. \end{proof}

\begin{rem} \label{rem:Hessian} For \emph{any} $X$, formula \eqref{Hessianformula} defines a bilinear form $H$ on $\so(n)$, but it is not generally symmetric when $X$ is not a critical point.  Symmetry when $X$ \emph{is} a critical point follows from general principles (equality of mixed second partials, in local coordinates) and can also be seen directly by computing \be H(E,N)-H(N,E) &=& \Tr( AXENBX^T+AXBNEX^T \\ & &  - AXNEBX^T - AXBENX^T) \\ &= & \Tr( AX[E,N]BX^T-AXB[E,N]X^T ) \\ &=& \Tr(AX[[E,N],B]X^T) \\ &=& (Df)(X)([E,N]), \ee using the fact that $E,N \in \so(n)$ are skew-symmetric.  \end{rem}

\begin{defn} \label{defn:Epq}  For $1 \leq p < q \leq n$, let $\EE(p,q) \in \so(n)$ be the skew-symmetric $n \times n$ matrix whose $(p,q)$ entry is $-1$, whose $(q,p)$ entry is $1$, and whose other entries are zero.  We refer to the basis for $\so(n)$ consisting of the $\EE(p,q)$ as the \emph{standard basis} for $\so(n)$.  For distinct $p,q \in \{ 1,\dots, n\}$, let $\FF(p,q)$ be the symmetric $n \times n$ matrix whose $(p,q)$ and $(q,p)$ entries are $1$ and whose other entries are zero.  Note $\FF(p,q) = \FF(q,p)$.  For $p \in \{ 1,\dots,n \}$, let $\DD(p)$ be the $n \times n$ matrix whose $(p,p)$ entry is $1$ and whose other entries are zero.  For $\sigma \in \SG_n$, let $P_\sigma$ be the matrix whose $i^{\rm th}$ column is column $\sigma(i)$ of the identity matrix. \end{defn}

For later use, we record the following formulae involving the matrices of Definition~\ref{defn:Epq}: \bne{EijBcommutator} [ \EE(p,q),B ] & = &  (B_{pp}-B_{qq}) \FF(p,q) \\ \label{EijFklcommutator} [ \EE(p,q) , \FF(u,v) ] &=& \left \{ \begin{array}{lll} 0, & \quad & \{ p,q \} \cap \{ u,v \} = \emptyset \\ \FF(q,v), & & p=u, q \neq v \\ \FF(q,u), & & p=v, q \neq u \\ - \FF(p,u), & & q=v, p \neq u \\ -\FF(p,v), & & q=u, p \neq v \\ 2 \DD(q) - 2 \DD(p), & & \{ p,q \} = \{ u,v \} \end{array} \right . \\ \label{Psigmaconj} P_\sigma \DD(p) P_\sigma^T &=& \DD(\sigma(p)) \\ \label{lastformula} P_\sigma \FF(p,q) P_\sigma^T & = & \FF(\sigma(p),\sigma(q)). \ene

Our first application of these computations is to establish that $f$ is Morse-Bott:

\begin{thm} \label{thm:bott}  For $X \in F \subseteq O(n)$ and $M \in T_X O(n)$, the following are equivalent: \begin{enumerate}[label=(\roman*), ref=\roman*] \item \label{Bott1} $M \in T_X F \subseteq T_X O(n)$. \item \label{Bott2} $AMBX^T+AXBM^T$ is symmetric. \item \label{Bott3} $MBX^T+XBM^T = \Diag(W_1,\dots,W_s)$ with each $W_i$ an $m_i \times m_i$ symmetric matrix. \item \label{Bott4} $AMBX^T+AXBM^T = MBX^TA+XBM^TA$. \item \label{Bott5} $BX^TAM+BM^TAX = X^TAMB+M^TAXB$. \item \label{Bott6} $BX^TAM+BM^TAX$ is symmetric.  \item \label{Bott7} $X^TAM+M^TAX = \Diag(U_1,\dots,U_t)$ with each $U_j$ an $n_j \times n_j$ symmetric matrix. \item \label{Bott8} $M \in \Ker H(X)$. \end{enumerate} In particular, the equivalence of \eqref{Bott1} and \eqref{Bott8} implies that $f$ is Morse-Bott. \end{thm}

\begin{proof}  According to Lemma~\ref{lem:nabla}, \be F &=& \{ X \in O(n) : AXBX^T = XBX^TA \}. \ee  Therefore, a tangent vector $M \in T_X O(n)$ will be in $T_X F$ iff $A(X+\epsilon M)B(X^T+\epsilon M^T)$ (viewed as a matrix with entries in $\RR[\epsilon]/\epsilon^2$) is symmetric.  Since $AXBX^T$ is symmetric (because $X \in F$), this is equivalent to \eqref{Bott2}.  This proves the equivalence of \eqref{Bott1} and \eqref{Bott2}.  

The conditions \eqref{Bott2}, \eqref{Bott3}, and \eqref{Bott4} are equivalent by Lemma~\ref{lem:1}, applied to the symmetric matrix $S=MBX^T+XBM^T$.  Similarly, \eqref{Bott5}, \eqref{Bott6}, and \eqref{Bott7} are equivalent by Lemma~\ref{lem:1}, applied to the symmetric matrix $X^TAM+M^TAX$ (and with $A$ there given by $B$ here).

Using the description of $H(X)$ in \eqref{Ddfver}, we see that \eqref{Bott8} is equivalent to \be \Tr(AM[ \EE(p,q),B]X^T+AX[ \EE(p,q),B]M^T) &=& 0 \ee for $1 \leq p < q \leq n$.  For any $n \times n$ matrices $C$ and $D$, we compute \be \Tr(C \FF(p,q) D) &=& \sum_{r=1}^n (C \FF(p,q) D)_{rr} \\ \nonumber &=& \sum_{r=1}^n \sum_{u,v=1}^n C_{ru} \FF(p,q)_{uv} D_{vr} \\ \nonumber &=& \sum_{r=1}^n C_{rp}D_{qr}+C_{rq}D_{pr} \\ \nonumber &=& (DC)_{qp}+(DC)_{pq}. \ee  Using  this and \eqref{EijBcommutator} we compute \be & & \Tr(AM[ \EE(p,q),B]X^T+AX[ \EE(p,q),B]M^T) \\ &=& (B_{pp}-B_{qq}) ( (X^TAM)_{pq}+(X^TAM)_{qp}+(M^TAX)_{pq}+(M^TAX)_{qp} ) \\ &=& 2(B_{pp}-B_{qq})(X^TAM+M^TAX)_{pq}. \ee  This is zero for $1 \leq p < q \leq n$ iff $X^TAM+M^TAX$ has the block form in \eqref{Bott7} (the blocks must be symmetric since $X^TAM+M^TAX$ is symmetric).  This proves the equivalence of \eqref{Bott8} and \eqref{Bott7}.

It remains only to prove the equivalence of \eqref{Bott4} and \eqref{Bott5}.  The equation in \eqref{Bott5} is equivalent to the equation \bne{Bott5A} XBX^TAMX^T+XBM^TA &=& AMBX^T + XM^TAXBX^T \ene obtained by multiplying on the left by $X$ and on the right by $X^T$.  Since $X \in F$, we have $AXBX^T=XBX^TA$ (Lemma~\ref{lem:nabla}), so we can rewrite \eqref{Bott5A} as \bne{Bott5B} AXBX^TMX^T+XBM^TA &=& AMBX^T+XM^TXBX^TA. \ene  Since $M \in T_X O(n)$ we have $X^TM=-M^TX$, so we can rewrite \eqref{Bott5B} as \be -AXBM^T+XBM^TA &=& AMBX^T-MBX^TA \ee or, equivalently, as \be XBM^TA + MBX^TA &=& AXBM^T+AMBX^T. \ee This equation  is equivalent to its transpose, which is nothing but the equation in \eqref{Bott4}.  The proof is complete. \end{proof}

Finally, in the remainder of this section, we will use the computations \eqref{EijBcommutator}-\eqref{lastformula} to determine the index of the Hessian of $f$.  In order to  express our results more simply we will assume (again, with no significant loss in generality) that the diagonal entries of $A$ and $B$ are in non-decreasing order---i.e.\ $a_1<\cdots<a_s$ and $b_1<\cdots<b_t$.

\begin{defn} \label{defn:signedpermutationmatrix} A \emph{sign matrix} is a diagonal matrix with diagonal entries in $\{ \pm 1\}$.  A \emph{signed permutation matrix (SPM)} is a matrix of the form $SP_{\sigma}$ with $S$ a sign matrix and $\sigma \in \SG_n$ a permutation. \end{defn}

\begin{lem} \label{lem:index} Every connected component of $F$ contains at least one signed permutation matrix. \end{lem}

\begin{proof} Each connected component of any orthogonal group certainly contains a SPM, so by Theorem~\ref{thm:main} it suffices to show that if $Q,R,\epsilon$ are as in Construction~\ref{const:main}, and the $Q[i]$ and $R[j]$ are SPMs, then $X = \Phi_{\epsilon}(Q,R)$ is a SPM.  Since $X \in O(n)$ it suffices to show that no row of $X$ contains more than one non-zero entry.  To see this, suppose $X_{p,q},X_{p,r} \neq 0$ for some $p,q,r$.  According to the construction of $X = \Phi_{\epsilon}(Q,R)$, we would then have \be & X_{p,q} = X[i,j]_{\ov{p},\ov{q}} = \sum_s Q[i,j]_{\ov{p},s} R[i,j]_{\ov{q},s} \\ & X_{p,r} = X[i,l]_{\ov{p},\ov{r}} = \sum_t Q[i,l]_{\ov{p},t} R[i,l]_{\ov{r},t} \ee for $i,j,l,\ov{p},\ov{q},\ov{r}$ determined from $p,q,r$ in the obvious manner.  Since $Q[i]$ is a SPM there is a unique $\ov{j}$ such that row $\ov{p}$ of $Q[i,\ov{j}]$ is not identically zero, and, furthermore, there is a unique $\ov{s}$ so that $Q[i,\ov{j}]_{\ov{p},\ov{s}} \neq 0$.  Since $X_{p,q},X_{p,r} \neq 0$, the former fact implies that $j=l = \ov{j}$, and then the latter fact (together with the fact that $R[i,\ov{j}]_{u,\ov{s}} \neq 0$ for at most one $u$ because $R[\ov{j}]$ is a SPM) implies that $\ov{q}=\ov{r}$.  Since $j=l$ and $\ov{q}=\ov{r}$ we have $q=r$. \end{proof}

\begin{thm} \label{thm:index} Let $X = SP_{\sigma}$ be a signed permutation matrix.  Then: \begin{enumerate} \item \label{index1} $X$ is a critical point of $f$.  \item \label{index2} If we let $\epsilon_{ij}$ denote the number of non-zero entries in the $m_i \times n_j$ block $X[i,j]$ of $X$ then $\epsilon=(\epsilon_{ij})$ is a perfect filling and there are $Q,R$ as in Construction~\ref{const:main} such that $X=\Phi_{\epsilon}(Q,R)$.  (One can even arrange that the $Q[i]$ are SPMs and the $R[j]$ are permutation matrices.)  \item \label{index3} The Hessian $H = H(f,X)$ of $f$ at $X$ (viewed as a quadratic form on $\so(n)$ as in Lemma~\ref{lem:Hessian}) is diagonal in the standard basis for $\so(n)$ and the numbers $H(p,q) := H( \EE(p,q),\EE(p,q) )$ (for $1 \leq p < q \leq n$) are given by \be H(p,q) & = & 2(B_{pp}-B_{qq})(A_{\sigma(q) \sigma(q)}-A_{\sigma(p) \sigma(p)}). \ee \item \label{index4}  The index of $H = H(f,X)$ is equal to the number of $(p,q) \in \{ 1, \dots,n \}^2$ with $B_{pp} < B_{qq}$ and $A_{\sigma(q) \sigma(q)} >  A_{\sigma(p) \sigma(p)}$.  In terms of the perfect filling $\epsilon$ associated to $X$ in \eqref{index2}, this index can be written $$ \sum_{(i,j) < (k,l)} \epsilon_{ij} \epsilon_{kl},$$ where $(i,j) < (k,l)$ means $i<k$ and $j < l$. \end{enumerate} \end{thm}  

\begin{proof} \eqref{index1} follows from \eqref{index2} in light of Proposition~\ref{prop:X}, though it can be seen more directly as follows:  Using \eqref{Psigmaconj} we compute $$X B X^T = SP_{\sigma} B P_\sigma^T S = \sum_{p=1}^n B_{pp} \DD(\sigma(p)).$$  Since this matrix is diagonal, $X \in F$ by Lemma~\ref{lem:nabla}.  

For \eqref{index2}, it is obvious from the construction of the $\epsilon_{ij}$ that $\epsilon$ is a perfect filling.  (For example, $\sum_j \epsilon_{ij}=m_i$ for any fixed $i$ because $\sum_j \epsilon_{ij}$ is just the total number of non-zero entries in some $m_i$ rows of $X$; this is $m_i$ because $X$ is a SPM.)  Fix some $i$ and $j$.  Let $z_1,\dots,z_r \in \{ 1, \dots, n_j \}$ ($r := n_j - \epsilon_{ij}$) be the identically zero columns of the $m_i \times n_j$ matrix $X[i,j]$.  Let $Q[i,j]$ be the $m_i \times \epsilon_{ij}$ matrix obtained from $X[i,j]$ by deleting these columns.  Let $R[i,j]^T$ be the $\epsilon_{ij} \times n_j$ matrix such that columns $z_1,\dots,z_r$ of $R[i,j]^T$ are identically zero and such that the $\epsilon_{ij} \times \epsilon_{ij}$ matrix obtained by deleting these zero columns from $R[i,j]^T$ is the identity matrix $I_{\epsilon_{ij}}$.  Then we clearly have \be X[i,j] &=& Q[i,j]R[i,j]^T. \ee For any fixed $i$, the $m_i \times m_i$ matrix $$Q[i] = \bp Q[i,1] & \cdots & Q[i,t] \ep $$ is obtained from the SPM $X$ by taking the $m_i$ rows $$ \bp X[i,1] & \cdots & X[i,t] \ep $$ of $X$ and then deleting the identically zero columns in the resulting matrix.  It is evident from this description of $Q[i]$ that $Q[i]$ is itself a SPM.  One can see similarly that, for any fixed $j$, the $n_j \times n_j$ matrix $$R[j] = \bp R[1,j] & \cdots & R[s,j] \ep $$ is a permutation matrix.

The description of the Hessian in \eqref{index3} follows from \eqref{Hessianformula} by using \eqref{EijBcommutator}-\eqref{lastformula}.  (The basic point here is that the diagonal entries of the matrices $\FF(p,q)$ are zero, hence the diagonal entries of any product of $\FF(p,q)$ and a diagonal matrix will also be zero.)

The first formula for the index in \eqref{index4} is immediate from \eqref{index3}.  (Note that $B_{pp} < B_{qq}$ can occur only if $p<q$ because we are assuming $b_1<\cdots<b_t$.)  To see that this is equal to the other formula for the index, first notice that the sum in \eqref{index4} counts the number of pairs $((p',p),(q',q))$ such that the $(p',p)$ and $(q',q)$ entries of $X$ are non-zero and the $(p',p)$ entry lies in a block $X[i,j]$ further up and to the left in $X$ than the block $X[k,l]$ containing $(q',q)$.  Since $X=S P_{\sigma}$ is a signed permutation matrix we must have $p' = \sigma(p)$ and $q' = \sigma(q)$ for any such pair.  Furthermore, the condition that the block $X[k,l]$ containing $(\sigma(q),q)$ is further down and further right than the block $X[i,j]$ containing $(\sigma(p),p)$ is equivalent to the conditions on $(p,q)$ in the first formula for the index.  \end{proof} 

Let $F_\epsilon \subseteq F$ be the image of the map $\Phi_{\epsilon}$, so that $F$ is the disjoint union of the $F_\epsilon$ by Theorem~\ref{thm:main}.  Then:

\begin{cor} \label{cor:index} The index of $H(f)$ on $F_{\epsilon}$ is constant, given by the sum in Theorem~\ref{thm:index}\eqref{index4}. \end{cor}

\begin{proof} The index of the Hessian is always locally constant, so it suffices to show that the index is given by the claimed formula on each component of $F_{\epsilon}$.  By Lemma~\ref{lem:index} each component of $F_{\epsilon}$ contains a signed permutation matrix and by Theorem~\ref{thm:index} the index at any such permutation matrix is as claimed. \end{proof}

\begin{example} \label{example:distincteigenvalues2}  Suppose $A=\Diag(a_1,\dots,a_n)$ (for distinct $a_p$) and $B = \Diag(b_1,\dots,b_n)$ (for distinct $b_p$).  (We continue to assume that the $a_p$ and the $b_p$ are in increasing order.)  In this case, Theorems~\ref{thm:main} and \ref{thm:bott} show that $f$ is a Morse function whose critical points are \emph{precisely} the signed permutation matrices $SP_{\sigma}$.  Fix a signed permutation matrix $X=SP_{\sigma}$.  In this case $B_{pp}<B_{qq}$ (resp. $A_{\sigma(q)\sigma(q)} > A_{\sigma(p) \sigma(p)}$) iff $p<q$ (resp.\ $\sigma(q) > \sigma(p)$), so by Theorem~\ref{thm:index} the index of $f$ at $X$ is given by $$| \{ (p,q): 1 \leq p < q \leq n {\rm \; and \;} \sigma(q) > \sigma(p) \} |$$ 
which is nothing but the inversion number of $\sigma$.
\end{example}

\begin{example} Another extreme example occurs when $A=aI_n$ and $B=bI_n$.  In this case $f$ is constant, so its index at any critical point is zero.  This is consistent with Theorem~\ref{thm:index}\eqref{index4} because there are no such pairs $(p,q)$, nor are there any $(i,j),(k,l)$ with $(i,j)<(k,l)$.  (Indeed, there is only one pair $(i,j)$, namely $(1,1)$.) \end{example}

\section{Linear Morse-Bott functions on $O(n)$} \label{section:SOn1}

The manifold $O(n)$ is a submanifold of the vector space $\MM$ of $n \times n$ matrices.  We call a function $f : O(n) \to \RR$ \emph{linear} if $f$ is the restriction of a linear function $\MM \to \RR$.  Since $\langle A,X \rangle := \Tr(A^TX)$ is a non-degenerate inner product on $\MM$, any linear function $f : \MM \to \RR$ is of the form $f(X):=\Tr(A^T X)$ for a unique $A \in \MM$.  Being linear, the function $f=f_A$ (whether regarded as a function on $\MM$ or on $O(n)$) is ``its own derivative" in the sense that \bne{df} (df)(X)(M) &=& \Tr(A^TM) \ene for all $X \in \MM$ (resp.\ $X \in O(n)$) and all $M \in T_X \MM = \MM$ (resp.\ $M \in T_X O(n)$).

We view the (positive definite) inner product on $\MM$ defined above as a ``constant" Riemannian metric on $\MM$.  Its restriction to $O(n)$ is (up to a positive constant) the Killing metric on $O(n)$.  The gradient $\nabla f$ of $f=f_A : \MM \to \RR$ with respect to this metric is characterized by the equality  \bne{gradchar} \langle (\nabla f)(X),M \rangle &=& (df)(X)(M) \ene for all $X \in \MM$, $M \in T_X \MM = \MM$.  Comparing \eqref{gradchar} and \eqref{df}, we find $(\nabla f)(X)=A$ for every $X \in \MM$.  Since the metric on $O(n) \subseteq \MM$ is the restriction of the metric on $\MM$, the gradient of $f|O(n)$ at $X \in O(n)$ is given by orthogonally projecting $(\nabla f)(X)=A \in T_X \MM$ onto the subspace $T_X O(n) \subseteq \MM$.  We thus find that \bne{gradfSOn} (\nabla f)(X) &=& \frac{1}{2}(A-XA^TX) \in T_X O(n) \ene for $X \in O(n)$, when $f$ is regarded as a function $f : O(n) \to \RR$.  Formula \eqref{gradfSOn} can also be obtained by directly verifying that, \emph{if} $(\nabla f)(X)$ \emph{were defined} by \eqref{gradfSOn}, then the equality \eqref{gradchar} characterizing the gradient would hold for all $X \in O(n)$, $M \in T_X O(n)$.  Since the critical points of $f$ are precisely the points of $X \in O(n)$ where $(\nabla f)(X)=0$, we obtain the following from \eqref{gradfSOn}:

\begin{lem} \label{lem:lineargradient} $X \in O(n)$ is a critical point of $f_A$ iff $XA^T$ is symmetric (equivalently $A^TX$ is symmetric). \end{lem}

The linear functions $f_A : O(n) \to \RR$ were studied in \cite{SS}, where it is shown that $f_A$ is a Morse function on $O(n)$ iff the symmetric matrix $AA^T$ has $n$ distinct eigenvalues.  

In the special case where $A = I$, the function $g := f_A$ is given by $g(X) = \Tr(X)$.  The function $g$ (or, more precisely, its restriction to $SO(n)$) was studied by Frankel in \cite{F}.  From Lemma~\ref{lem:lineargradient} we see that the critical locus $F$ of $g$ is given by $$F = \{ X \in O(n) : X = X^T \} = \{ X \in O(n) : X^2 = I \}.$$  Frankel's study of $g$ in \cite{F} makes heavy use of the observation that $g$ is a class function (i.e.\ is conjugation invariant).  This is an excellent technique in that it allows him to bring to bear the methods of Lie theory:  The study of $g$ is thus ultimately reduced to the study of the restriction of $g$ to ``the" maximal torus $T$.\footnote{To do this one must always have some understanding of how the normalizer of ``the" maximal torus $T$ acts on $T$ by conjugation.  In \cite{F} this amounts to knowing the centralizers in $SO(n)$ of certain elements of the usual maximal torus $T \subseteq SO(n)$ and to knowing when certain elements of $T$ are conjugate to each other in $SO(n)$.}  This technique also allows Frankel to treat the other ``classical groups" $U(n), Sp(n)$ and their ``trace functions" (or rather, their real parts) on the same footing and also lends itself to Frankel's study of Stiefel manifolds in the second half of \cite{F}.

Despite these advantages of Frankel's technique, we would like to point out that many of the proofs in \cite{F} (at least in the case of $SO(n)$) can be greatly simplified by using more direct arguments.  For example, the description of $F$ obtained above is about as simple as one could imagine.  In \cite{F}, however, this description of $F$ is obtained as follows: \begin{enumerate} \item \label{F1} Frankel first shows \cite[Lemma~1]{F} that $(\nabla g)(X) \in T_X T \subseteq T_X SO(n)$ for all $X \in T \subseteq SO(n)$.  (This alone takes a page in \cite{F}, though it can be seen almost immediately from \eqref{gradfSOn}.)  \item \label{F2} From \eqref{F1} it follows that $X \in T$ is in $F$ iff $X$ is a critical point of $g|T$.  By making use of the explicit description of the ``usual maximal torus" $T$ Frankel now directly verifies that $X \in T$ is in $F$ iff $X^2=I$.  (This of course also takes a while since he divides into cases depending on the parity of $n$.)  \item \label{F3} Finally, by making use of the fact that any $X \in SO(n)$ is conjugate in $SO(n)$ to some $X' \in T$, the fact that $g$ is conjugation invariant (hence the condition ``$X \in F$" is conjugation invariant), and the fact that the condition $X^2=I$ is conjugation invariant, one arrives at the description of $F$. \end{enumerate}

In \cite{F} Frankel next establishes a diffeomorphism \bne{Fdiffeo} \coprod_k \GG(2k,n) & \to & F \subseteq SO(n) \\ \nonumber \Lambda & \mapsto & X(\Lambda), \ene where $\GG(2k,n)$ is the Grassmannian of $2k$-dimensional linear subspaces $\Lambda \subseteq \RR^n$ and $F$ now denotes the critical locus of $g : SO(n) \to \RR$.  The description of this diffeomorphism in \cite{F} is again rather circuitous.  It can be obtained easily as follows:  Given $\Lambda \in \GG(2k,n)$, let $X = X(\Lambda)$ be the unique linear transformation $\RR^n \to \RR^n$ such that $\Lambda$ (resp.\ $\Lambda^\perp$) is the $(-1)$-eigenspace (resp.\ $1$-eigenspace) of $X$.  Obviously $X$ is self-adjoint (so $X=X^T$), $X^2=I$, and $X \in SO(n)$ (because the dimension $2k$ of $\Lambda$ is even), so $X \in F$.  This yields a map as in \eqref{Fdiffeo} which is clearly smooth.  A smooth inverse for this map can be constructed as follows:  Fix $X \in F$.  Since $X$ is symmetric, it is (orthogonally) diagonalizable (over $\RR$, so its eigenvalues are real) and distinct eigenspaces of $X$ are orthogonal.  Furthermore $X \in O(n)$, so its eigenvalues have magnitude $1$.  Therefore, if we let $\Lambda = \Lambda(X)$ be the $(-1)$-eigenspace of $X$, then we have an orthogonal direct sum decomposition $\RR^n = \Lambda \oplus \Lambda^\perp$ with $\Lambda^\perp$ equal to the $1$-eigenspace of $X$.  Since $X \in SO(n)$, the dimension of $\Lambda$ must be even ($2k$, say).  Evidently $X \mapsto \Lambda(X)$ is the inverse of $\Lambda \mapsto X(\Lambda)$.\footnote{The diffeomorphisms \eqref{Fdiffeo} are written differently in \cite{F} because Frankel identifies a critical point with its $1$-eigenspace, rather than with its $(-1)$-eigenspace as we have done above.}  If we instead work with $O(n)$, then the same discussion (with all parity considerations removed) yields a diffeomorphism \bne{Fdiffeo2} \coprod_k \GG(k,n) & \to & F \subseteq O(n). \ene

The Hessian of $g$ (and, more generally, of any linear function $f_A : O(n) \to \RR$) is easily described.  As in \S\ref{section:BMF1}, we view the Hessian of $f_A$ as a quadratic form on $\so(n)$ via the isomorphism \eqref{TangentSpaceIso}.

\begin{lem} \label{lem:linearH}  The Hessian $H=H(f_A,X)$ of the linear function $f_A$ at a critical point $X$, viewed as a quadratic form on $\so(n)$ as above, is given by \be H(E,N) &=& \Tr(A^TXEN) \ee for $E,N \in \so(n)$. \end{lem}

\begin{proof}  The proof is essentially the same as the proof of Lemma~\ref{lem:Hessian}.  Equation \eqref{dfver} there becomes \be (df)^{ver} : O(n) & \to & \so(n)^* \\ \nonumber (df)^{ver}(X)(N) &=& \Tr(A^TXN) \ee here.  Equation \eqref{Ddfver} there becomes \bne{Ddfverlinear} (D(df)^{ver})(X) : T_X O(n) & \to & \so(n)^* \\ \nonumber (D(df)^{ver})(X)(M)(N) & = & \Tr( A^TMN) \ene here.  Equation \eqref{Hessian} there becomes \bne{Hessianlinear} H(X) : \so(n) & \to & \so(n)^* \\ \nonumber  H(X)(E,N) & = & \Tr(A^TXEN) \ene here. \end{proof}

\begin{rem} \label{rem:linearH} (Cf.\ Remark~\ref{rem:Hessian})  The bilinear form $H$ of Lemma~\ref{lem:linearH} is not generally symmetric when $X \notin F$.  When $X \in F$ (so $A^TX=X^TA$ by Lemma~\ref{lem:lineargradient}) one can directly verify that $H$ is symmetric by computing \be \Tr(A^TXNE) &=& \Tr(A^TXN^TE^T) \\ &=& \Tr(ENX^TA) \\ &=& \Tr(ENA^TX) \\ &=& \Tr(A^TXEN), \ee using the fact that $E,N \in \so(n)$ and standard properties of the trace. \end{rem}

We can see from Lemma~\ref{lem:linearH} that $g : O(n) \to \RR$ is Morse-Bott, as follows:  Fix any $X \in F \subseteq O(n)$, $M \in T_X O(n)$.  Since $F$ consists of the symmetric matrices in $O(n)$ we have $M \in T_X F$ iff $X+\epsilon M \in SO(n,\RR[\epsilon]/\epsilon^2)$ is symmetric, which, since $X$ is symmetric, is equivalent to saying $M$ is symmetric.  It is obvious from basic properties of the trace that $\Tr(MN)=0$ whenever $M$ is symmetric and $N$ is skew-symmetric.  On dimension grounds we therefore have an orthogonal direct sum decomposition \be \MM &=& \{ M \in V : M=M^T \} \oplus \{ N \in V : N=-N^T \} \\ \nonumber &=& \{ M \in V : M=M^T \} \oplus \so(n). \ee  Therefore $M$ is in $T_X F$ iff $\Tr(MN)=0$ for all $N \in \so(n)$.  From \eqref{Ddfverlinear} (with $A=I$ there) we see that this latter condition is equivalent to $M$ being in the kernel of $H(g,X)$.  This proves that the kernel of the Hessian $H(g,X)$ is precisely $T_X F$ and therefore $g$ is Morse-Bott.

Since the Grassmannians $\GG(k,n)$ are connected, the index of the Hessian of $g$ is determined by its values at the critical points $-I_{k} \oplus I_{n-k}$.  These indices are easily computed:

\begin{lem} \label{lem:linearindex} When $X = -I_{k} \oplus I_{n-k}$ and $A=I$, the quadratic form $H(E,N) = \Tr(XEN)$ on $\so(n)$ of Lemma~\ref{lem:linearH} is diagonal in the standard basis (Definition~\ref{defn:Epq}) for $\so(n)$.  The numbers $H(p,q) := H(\EE(p,q),\EE(p,q))$ are given by \be H(p,q) & = & \left \{ \begin{array}{lll} -2, & \quad & k < p < q \leq n \\ 0, & & 1 \leq p \leq k < q \leq n \\ 2, & & 1 \leq p < q \leq k. \end{array} \right . \ee The index of $H$ is $\iota(k) := \bp n-k \\ 2 \ep$.\footnote{When $n=2m+1$ is odd, this formula for the index of $H$ is equivalent to \cite[Lemma~3]{F}.} \end{lem}

\begin{proof} For standard basis vectors $\EE(p,q), \EE(u,v) \in \so(n)$, we see that $\EE(p,q)\EE(u,v)$ has no non-zero diagonal entries when $(p,q) \neq (u,v)$.  Since the effect of multiplying on the left by $X$ is simply to multiply the first $k$ rows by $-1$, the matrices $X \EE(p,q) \EE(u,v)$ still have no non-zero diagonal entries when $(p,q) \neq (u,v)$.  These matrices therefore have trace zero, which shows that $H$ is diagonal in the standard basis.  The non-zero diagonal entries of $\EE(p,q)^2$ are precisely the $(p,p)$-entry and the $(q,q)$-entry, both of which are $-1$.  Taking into account the effect of multiplying on the left by $X$, we arrive at the formula for the $H(p,q)=\Tr(X \EE(p,q)^2)$.  The formula for the index amounts to counting the number of pairs $(p,q)$ with $k < p < q \leq n$. \end{proof}

The ``Morse-Bott inequalities" (i.e.\ the existence of a spectral sequence going from the cohomology of the critical locus $F$ shifted by the index of the Hessian to the cohomology of $SO(n)$) for $g$ imply that \bne{MI} \dim \H^i(SO(n),\FF) & \leq & \sum_k \dim \H^{i-\iota(2k)}(\GG(2k,n),\FF) \ene for all $i,n$ and any field $\FF$ (with $\iota(2k)$ as in Lemma~\ref{lem:linearindex}).  When $\FF=\FF_2$, Frankel shows in \cite{F} that the inequalities \eqref{MI} are in fact \emph{equalities}.  This is done by appeal to a result of E.~E.~Floyd \cite[Theorem~A]{F} asserting that, since $F$ is the fixed locus of the involution $X \mapsto X^{-1}=X^T$ of the smooth, compact manifold $SO(n)$, the sum of the mod $2$ Betti numbers of $F$ is bounded above by the sum of the mod $2$ Betti numbers of $SO(n)$.  We believe that it is possible to give a purely Morse theoretic proof of this fact (cf.\ \S\ref{section:SOn2}), but we have not attempted this.  In the Appendix we give a purely combinatorial proof of Frankel's mod $2$ Betti number relationship.

\begin{example} \label{example:SO3} When $n=3$ the critical locus $F$ of $g : SO(3) \to \RR$ is the disjoint union of $\{ I \} = \GG(0,3)$ (index $3$) and $\GG(2,3) \cong \RR \PP^2$ (index $0$).  The equalities of mod $2$ Betti numbers above amount to the equality of (mod $2$) Poincar\'e polynomials \be p(SO(3)) &=& p(\RR \PP^3) \\ &=& 1+t+t^2+t^3 \\ &=& p(\RR \PP^2) + t^3 p( \{ I \} ). \ee \end{example}

\section{Simple Morse-Bott cohomology} \label{section:BMC}

Let $X$ be a smooth compact manifold of dimension $d$.  For simplicity, we assume $X$ is connected.  A \emph{simple Morse-Bott function} is a non-constant Morse-Bott function $f : X \to \RR$ whose critical locus $F$ consists only of points where $f$ obtains its maximum or minimum value.  Throughout this section, $f$ will be a simple Morse-Bott function on $X$.

We have $F=F_0 \coprod F_k$, where $F_0$ (resp.\ $F_k$) is the index $0$ (resp.\ $k$) critical locus consisting of minima (resp.\ maxima) for $f$.  The image of $f$ must be a closed interval $[a,b] \subseteq \RR$ with $F_0 = f^{-1}(a)$, $F_k = f^{-1}(b)$.  Note that $k$ is also the codimension of $F_k$ in $X$, since the Hessian of $f$ must be negative definite on the normal bundle of $F_k$ in $X$.  Let $m$ be the codimension of $F_0$ in $X$.

Fix a Riemannian metric on $f$.  The gradient $\nabla f$ is the vector field on $X$ dual to the $1$-form $df$ under the metric.  Integration of $\nabla f$ yields a smooth action of the group $G = (\RR, +)$ on $X$, denoted $\tau \cdot x$.  The action fixes $F$ and is free on $X \setminus F$.  For any $x \in X \setminus F$, the function $t \mapsto f(\tau \cdot x)$ is a strictly increasing function of $\tau \in \RR$ approaching $b$ (resp.\ $a$) at $t \to \infty$ (resp.\ $\tau \to -\infty$).  The quotient $(X \setminus F) / G =: M$ can be (and will be) identified with any regular fiber $f^{-1}(c)$ ($c \in (a,b)$) of $f$, thus we view $M$ as a closed subspace of $X$ contained in $X \setminus F$.  There is a continuous source map \be s: X \setminus F_k & \to & F_0 \\ \nonumber s(x) & := & \lim_{\tau \to -\infty} \tau \cdot x \ee which is a locally trivial $\RR^m$ bundle and whose restriction to $M$ is a locally trivial $S^{m-1}$ bundle (any regular level set of $f$ intersects any fiber of $\tau$ in a sphere).  (See \cite[Theorem~A.9]{AB}.  It seems that $X \setminus F_k$ should be diffeomorphic to the normal bundle $N=N_{F_0/X}$ by a diffeomorphism exchanging $s$ and the projection $N \to F_0$ but in loc.\ cit.\ this is only shown to hold locally.  One can easily show that, in the situation we shall consider in \S\ref{section:SOn2}, one does have such a global diffeomorphism.)  Similarly, there is a continuous target map \be t: X \setminus F_0 & \to & F_k \\ \nonumber t(x) & := & \lim_{\tau \to +\infty} \tau \cdot x \ee which is a locally trivial $\RR^k$ bundle and whose restriction to $M$ is a locally trivial $S^{k-1}$ bundle.

Fix a ``coefficient" ring $A$.  The properties of $t$ mentioned above imply that $\R t_! \underline{A}_{X \setminus F_0}$ is locally isomorphic to $\underline{A}_{F_k}[-k]$.  We assume that $t$ is \emph{oriented} (with respect to $A$) in the sense that there is an isomorphism \be \eta : \R t_! \underline{A}_{X \setminus F_0} & \to & \underline{A}_{F_k}[-k] \ee (i.e.\ an isomorphism of $\underline{A}_{F_k}$ modules $\R^k t_! \underline{A}_{X \setminus F_0} \cong \underline{A}_{F_k}$).  Set $Y=f^{-1}[a,c]$, $Z=f^{-1}[c,b]$.

We have a commutative diagram of $\underline{A}_X$ modules with exact rows \bne{CDAXmodules} & \xym{ 0 \ar[r] & \underline{A}_{Z \setminus M} \ar[r] & \underline{A}_{Z} \ar[r] & \underline{A}_{M} \ar[r] & 0 \\ 0 \ar[r] & \underline{A}_{X \setminus Y} \ar@{=}[u] \ar[d] \ar[r] & \underline{A}_{X} \ar[u] \ar@{=}[d] \ar[r] & \underline{A}_{Y} \ar[u] \ar[d] \ar[r] & 0 \\ 0 \ar[r] & \underline{A}_{X \setminus F_0} \ar[r] & \underline{A}_{X} \ar[r] & \underline{A}_{F_0} \ar[r] & 0 } \ene where we have suppressed notation for proper pushforwards to $X$ (this is the usual pushforward for the closed subspaces $F_0$, $M$, $Y$, and $Z$ and the ``extension by zero" for the open subspaces $Z \setminus M = X \setminus Y$ and $X \setminus F_0$).

The map $\underline{A}_{X \setminus Y} \to \underline{A}_{X \setminus F_0}$ in \eqref{CDAXmodules} becomes an isomorphism when $\R t_!$ is applied:  Indeed, by the base change theorem for proper direct images, it suffices to show that, for any $x \in F_k$, the map $(t|X \setminus Y)^{-1}(x) \into t^{-1}(x)$ induces an isomorphism on compactly supported cohomology.  This map is homeomorphic to the inclusion of the open unit ball into $\RR^k$, so this is indeed the case.  The ``orientation" isomorphism $\eta$ from above therefore also yields an isomorphism $\R t_! \underline{A}_{X \setminus Y} \cong \underline{A}_{F_k}[-k]$ which we also call $\eta$.

Applying $\R t_!$ to (the triangle associated to) the top row of \eqref{CDAXmodules} and using this isomorphism, we obtain a triangle \bne{Fktriangle} & \xym{ \underline{A}_{F_k}[-k] \ar[r] & \R t_! \underline{A}_Z \ar[r] & \R t_* \underline{A}_M \ar[r]^-{t_*} & \underline{A}_{F_k}[1-k] } \ene in the derived category $\D( \underline{A}_{F_k} )$ (note that $\R t_! \underline{A}_M = \R t_* \underline{A}_M$ because $t|M : M \to F_k$ is a sphere bundle, so it is proper).  Applying $\R \Gamma = \R \Gamma(X,\slot)$ to \eqref{CDAXmodules} and using the isomorphism(s) $\eta$ yields a commutative diagram \be & \xym{ \R \Gamma(F_k,A)[-k]  \ar[r] & \R \Gamma(Z,A) \ar[r] & \R \Gamma(M,A) \ar[r]^-{t_*} & \R \Gamma(F_k,A)[1-k] \\ \R \Gamma(F_k,A)[-k] \ar@{=}[u] \ar[d]_{\cong} \ar[r] & \R \Gamma(X,A) \ar[u] \ar@{=}[d] \ar[r] & \R \Gamma(Y,A) \ar[u] \ar[d]_{\cong} \ar[r] & \R \Gamma(F_k,A)[1-k] \ar[d]_{\cong} \ar@{=}[u] \\ \R \Gamma_!(X \setminus F_0,A) \ar[r] & \R \Gamma(X,A) \ar[r] & \R \Gamma(F_0,A) \ar[r] & \R \Gamma_!(X \setminus F_0,A)[1]  } \ee in $\D(A)$ where the rows are triangles, $\Gamma_!$ is the compactly supported global sections functor, and the vertical arrows define maps of triangles.  The map $\R \Gamma(Y,A) \to \R \Gamma(F_0,A)$ is the map induced by the inclusion $F_0 \into Y$.  This inclusion induces isomorphisms on cohomology because it is retracted by the map $s|Y : Y \to F_0$, which is a closed disc bundle and hence induces isomorphisms on cohomology.  We thus see that the map $\R \Gamma(F_0,A) \to \R \Gamma(M,A)$ given by composing the inverse of $\R \Gamma(Y,A) \to \R \Gamma(F_0,A)$ and $\R \Gamma(Y,A) \to \R \Gamma(M,A)$ (this map is the one induced by the inclusion $M \into Y$) is nothing but the map induced by $s|M : M \to F_0$.  We thus obtain a triangle \be & \xym{ \R \Gamma(F_k,A)[-k] \ar[r] & \R \Gamma(X,A) \ar[r] & \R \Gamma(F_0,A) \ar[r]^-{ t_* s^*} & \R \Gamma(F_k,A)[1-k] } \ee in $\D(A)$ and hence also a long exact sequence \bne{LES} & \xym{ \cdots \ar[r]^-{t_*s^*} & \H^{i-k}(F_k) \ar[r] & \H^i(X) \ar[r] & \H^i(F_0) \ar[r]^-{t_* s^*} & \H^{i+1-k}(F_k) \ar[r] & \cdots } \ene in cohomology (coefficients in $A$ understood).

The map $t_* : \H^{*}(M) \to \H^{*+1-k}(F_k)$ appearing in \eqref{LES} satisfies the \emph{Projection Formula} \be  t_*(t^*(\alpha) \cdot \beta) & = & \alpha \cdot t_*( \beta) \ee for $\alpha \in \H^*(F_k)$, $\beta \in \H^*(M)$.  This holds as a matter of ``general nonsense" owing to the construction of the latter map from the map $t_* : \R t_* \underline{A}_M \to \underline{A}_{F_k}[1-k]$ in \eqref{Fktriangle}.  We will recall the details for the reader's convenience.  A cohomology class $\alpha \in \H^r(F_k)$ is the same thing as a $\D(\underline{A}_{F_k})$-morphism $\alpha : \underline{A}_{F_k} \to \underline{A}_{F_k}[r]$.  Similarly, $\beta \in \H^q(M)$ is a map $\beta : \underline{A}_M \to \underline{A}_M[q]$.  The cup product $t^*(\alpha) \cdot \beta$ corresponds to the composition \be & \xym@C+10pt{ \underline{A}_M \ar[r]^-{ t^{-1} \alpha } & \underline{A}_M[r] \ar[r]^-{\beta[r]} & \underline{A}_M[r+q].} \ee  The adjunction map $\Id \to t_* t^{-1}$ yields a natural transformation $\Id \to Rt_* t^{-1}$.  Evaluating this on $\underline{A}_{F_k}$ yields a natural map $\theta : \underline{A}_{F_k} \to  \R t_* \underline{A}_M$.  The cohomology class $t_* \beta \in \H^{q+1-k}(F_k)$ corresponds to the composition \be & \xym{ \underline{A}_{F_k} \ar[r]^-{\theta} & \R t_* \underline{A}_M \ar[r]^-{\R t_* \beta} & \R t_* \underline{A}_M[q] \ar[r]^-{t_*[q]} & \underline{A}_{F_k}[q+1-k]. } \ee

We have a commutative diagram \be & \xym@C+15pt{ \underline{A}_{F_k} \ar[r]^-{\alpha} \ar[d]_{\theta} & \underline{A}_{F_k}[r] \ar[d]_{\theta[r]} & \underline{A}_{F_k}[q+r+1-k] \\ \R t_* \underline{A}_M \ar[r]^-{ \R t_* t^{-1} \alpha } & \R t_* \underline{A}_M[r] \ar[r]^-{ \R t_* \beta[r] } & \R t_* \underline{A}_M[q+r] \ar[u]^{t_*[q+r]} } \ee in $\D(\underline{A}_{F_k})$.  The two ways around this diagram are the two sides of the Projection Formula (right first is the LHS, down first is the RHS).

The long exact sequence \eqref{LES} gives rise to short exact sequences \be  0 \to \Cok(\delta_{i-1}) \to \H^i(X) \to \Ker(\delta_i) \to 0, \ee where \be \delta_i = t_*s^* : \H^i(F_0) & \to & \H^{i+1-k}(F_k). \ee

\section{A particularly nice Morse-Bott function on $SO(n)$} \label{section:SOn2}

We now specialize the discussion of \S\ref{section:SOn1} to the case where $A = {\rm Diag}(0,\dots,0,1)$.  In this case $f=f_A : SO(n) \to \RR$ is given by $f(X) = X_{nn}$ (the lower right entry of $X$).  This function $f$ arises as the composition of the map \be p : SO(n) & \to & S^{n-1} \\ p(X) & := & (X_{n1},\dots,X_{nn}) \ee and the usual height function \be h :S^{n-1} & \to & \RR \\ h(x_1,\dots,x_n) & := & x_n. \ee  The map $p$ is submersive --- in fact it is an $SO(n-1)$ principal bundle (see below).  The height function is Morse, hence $f$ is Morse-Bott.  The critical points of the height function are just the points $(0,\dots,0,\pm 1)$ of $S^{n-1}$ where the height is minimum or maximum, hence the critical loci of $f$ are just the points of $SO(n)$ mapped to these two points by $p$.

Explicitly, the critical locus $F$ of $f$ is $F_0 \coprod F_{n-1}$ where \be F_0 & = & \{ X \in SO(n) : X_{nn} = -1 \} \\ F_{n-1} & = & \{ X \in SO(n) : X_{nn}=1 \}. \ee  We identify $F_{n-1}$ with $SO(n-1)$ by taking $Q \in SO(n-1)$ to the block-diagonal matrix $Q \oplus 1 \in F_{n-1}$.  Throughout, we let $J := {\rm Diag}(-1,1,\dots,1) \in O(n)$.  Geometrically, $J$ is a reflection across the hyperplane $e_1^\perp$.  We have $\det J = -1$ and $J^2=I$.  We identify $F_0$ with $SO(n-1)$ by taking $Q \in SO(n-1)$ to the block-diagonal matrix $JQ \oplus -1 \in F_0$.

Since $f$ takes its minimum (resp.\ maximum) value on $F_0$ (resp.\ $F_{n-1}$), the Hessian of $f$ must be positive (resp.\ negative) definite on the normal bundle of $F_0$ (resp.\ $F_{n-1}$).  Therefore, $F_0$ is of index $0$ and $F_{n-1}$ is of index $n-1 = {\rm codim}(F_{n-1} \subseteq SO(n))$.

The moduli space of flow lines $M = f^{-1}(0)$ is just the set of $X \in SO(n)$ with $X_{nn}=0$.  We shall make use of the map $\pi : M \to S^{n-2}$ defined by $\pi(X) := (X_{1n},\dots,X_{n-1,n})$.  I.e., $\pi(X)$ is the \emph{right column} of $X$, excepting the lower right entry, which is zero.  The map $\pi$ is not to be confused with the restriction of $p : SO(n) \to S^{n-1}$ to $M$.  The latter can also be viewed as a map $M \to S^{n-2}$, defined using the \emph{bottom row} rather than the \emph{right column} (always excepting the lower right entry).

The group $G := SO(n-1)$ acts (on the left, smoothly) on $SO(n)$ by letting $g \in G$ act on $SO(n)$ via left multiplication by $g \oplus 1 \in SO(n)$.  This is an action through isometries of $SO(n)$ making $p$ a principal $G$-bundle.  In particular, $p$ is $G$-invariant, and hence so is $f=hp$.  Hence $G$ also acts naturally on $F_0$, $F_{n-1}$, and $M$, making the source map $s:M \to F_0$ and target map $t : M \to F_{n-1}$ equivariant.  Under the identifications $F_0 = SO(n-1)$, $F_{n-1} = SO(n-1)$ from above, the $G$ action on $F_{n-1}$ is identified with the action of $SO(n-1)$ on itself by left multiplication ($g \cdot Q = gQ$), while the $G$ action on $F_0$ is identified with the action of $SO(n-1)$ on itself defined by $g \cdot Q = JgJQ$.  Note that the map $\pi : M \to S^{n-2}$ defined above is $G$-\emph{equivariant} (not $G$-\emph{invariant}) when $G=SO(n-1)$ acts on $S^{n-2}$ by left multiplication, as usual.  The latter action is transitive, hence:

\noindent {\bf (*)}  For any $X \in M$, there is a $g \in G$ such that the right column of $g \cdot X$ is $e_1$.

The general formula \eqref{gradfSOn}, specialized to our case ($A={\rm Diag}(0,\dots,0,1)$) yields the following explicit formula for the gradient of our Morse-Bott function $f$: \be (\nabla f)(X) & = & \frac{1}{2} \bp -X_{1n}X_{n1} & -X_{1n}X_{n2} & \cdots & -X_{1n}X_{nn} \\
                                                                       -X_{2n}X_{n1} & -X_{2n}X_{n2} & \cdots & -X_{2n}X_{nn} \\
																	   \vdots & \vdots & & \vdots \\
																	   -X_{n-1,n}X_{n1} & -X_{n-1,n}X_{n2} & \cdots & -X_{n-1,n}X_{nn} \\
																	   -X_{nn}X_{n1} & -X_{nn}X_{n2} & \cdots & 1-X_{nn}^2 \ep. \ee  
The \emph{key observation} about this formula is that rows $2,\dots,n-1$ of $(\nabla f)(X)$ will be zero whenever $X$ has $X_{2n}=X_{3n}=\cdots=X_{n-1,n}=0$.  This implies that, for such an $X$, rows $2,\dots,n-1$ of $X$ must remain constant along the gradient flow of $X$ (and hence also at the two limit points of the flow of $X$, when $X$ is not a critical point).  This observation and {\bf (*)} above allow us to describe the source and target maps $s$, $t$ explicitly, as follows:  Consider some $X \in M$ whose right column is $e_1$.  Write $X$ in the block form \bne{blockform} X & = & \bp 0 & 1 \\ V & 0 \\ v & 0 \ep \ene where $V$ is $(n-2) \times (n-1)$ and $v$ is $1 \times (n-1)$.  Note that $$ \bp V \\ v \ep $$ is in $O(n-1)$ and has determinant $(-1)^{n+1}$, hence \be \det \bp v \\ V \ep & = & -1. \ee  Since rows $2,\dots,n-1$ of $s(X) \in F_0 \subseteq SO(n)$ must be the same as those of $X$, we must have \be & s(X) = \bp v & 0 \\ V & 0 \\ 0 & -1 \ep \in F_0 \subseteq SO(n) \ee Under our identification $F_0 = SO(n-1)$, we therefore have \bne{sX2} s(X) & = & \bp -v \\ V \ep \in SO(n-1) \ene (recall that this identification involves a left multiplication by $J$, which changes the sign of the first row).  Similar considerations show that, for $X$ as above, $t(X) \in F_{n-1} \subseteq SO(n)$ is given by \be & t(X) = \bp -v & 0 \\ V & 0 \\ 0 & 1 \ep \in F_{n-1} \subseteq SO(n). \ee Under our identification $F_{n-1}=SO(n-1)$, we therefore have \bne{tX2} t(X) & = & \bp -v \\ V \ep  \in SO(n-1). \ene  Since the maps $s : M \to SO(n-1)$ and $t : M \to SO(n-1)$ are $G$-equivariant, they are determined by \eqref{sX2} and \eqref{tX2} in light of {\bf (*)}.

The maps $s$ and $t$, while not the same, are still ``close to being equal" in a sense we will now make precise.  Since $SO(n-1)$ is a Lie group, the set $\Hom(M,SO(n-1))$ has a natural group structure, hence we can consider the map $st^{-1} : M \to SO(n-1)$.  (If $s$ and $t$ are ``kind of the same," then $st^{-1}$ ought to be ``close to the identity".)  To describe this map, first notice that, for any $L \in \RR \PP^{n-2}$ (i.e.\ any $1$-dimensional linear subspace of $\RR^{n-1}$), the reflection $\rho(L^\perp)$ across the hyperplane $L^\perp$ is an orthogonal linear transformation of $\RR^{n-1}$ of determinant $-1$.  Following this reflection by our reflection $J$ yields an element $r(L) := J \rho(L^\perp) \in SO(n-1)$.  This defines a smooth map $r : \RR \PP^{n-2} \to SO(n-1)$.

\begin{prop} \label{prop:main} The map $st^{-1} : M \to SO(n-1)$ is equal to the composition of the map $\pi : M \to S^{n-2}$, the quotient projection $q : S^{n-2} \to \RR \PP^{n-2}$, and the map $r : \RR \PP^{n-2} \to SO(n-1)$ discussed above. \end{prop}

\begin{proof}  Given any $Y \in M$, by {\bf (*)} we can find some $g \in G=SO(n-1)$ and some $X \in M$ with right column $e_1$ so that $Y = g \cdot X$.  If we write this $X$ in the block form \eqref{blockform}, then, as we saw above, \be & s(Y) = s(g \cdot X) = g \cdot s(X) = JgJ  \bp -v \\ V \ep \\ & t(Y) = t(g \cdot X) = g \cdot t(X) = g  \bp -v \\ V \ep, \ee hence $(st^{-1})(Y) = JgJg^{-1} = JgJg^T$.  Now write $g$ in the block form \be g &=& \bp w & W \ep \ee where $w$ is $(n-1) \times 1$ and $W$ is $(n-1) \times (n-2)$.  The fact that $gg^T=I$ is equivalent to $WW^T = I-ww^T$, and, using this, we compute \be gJg^T & = & \bp w & W \ep J \bp w^T \\ W^T \ep \\ & = & \bp w & W \ep \bp -w^T \\ W^T \ep \\ & = & -ww^T+WW^T \\ & = & I - 2 ww^T. \ee  Since the matrix $ww^T$ is the orthogonal projection onto (the span of) $w$, we see that $I-2ww^T = \rho(w^\perp)$ is the reflection across $w^\perp$.  We have $$ Y = g \cdot X = \bp w & W & 0 \\ 0 & 0 & 1 \ep \bp 0 & 1 \\ V & 0 \\ v & 0 \ep = \bp WV & w \\ v & 0 \ep,$$ so $\pi(Y)=w$, $(q\pi)(y)$ is the span of $w$, and $(rq\pi)(y) = J \rho(w^\perp)$ is reflection across $w^\perp$ followed by $J$, which is indeed equal to $(st^{-1})(Y) = JgJg^{-1}$ since we just saw that $gJg^{-1} = \rho(w^\perp)$. \end{proof} 

We can use our study of the Morse-Bott function $f$ to calculate the Betti numbers $b_i(n) := \dim \H^i(SO(n),\FF_2)$ of $SO(n)$ with coefficients in the two element field $\FF_2$.  We will show that these are given by \be b_i(n) & = &  |   \{ S \subseteq \{ 1, \dots, n-1 \} : \sum_{s \in S} s = i  \}  |. \ee  Using the fact that a subset $S \subseteq \{ 1, \dots, n-1 \}$ either contains $n-1$ or does not, we obtain \bne{biformula2} b_i(n) & = & b_i(n-1) + b_{i+1-n}(n-1) \ene for all $n>1$.  The formula \eqref{biformula2}, together with the fact that $b_0(1)=1$ and $b_i(1)=0$ for $i \neq 0$, uniquely determines all the $b_i(n)$.  In terms of the Poincar\'e polynomials $p_n(t) := \sum_i b_i(n)t^i$, the description of the $b_i(1)$ above is equivalent to $p_1(t)=1$ and formula \eqref{biformula2} is equivalent to \bne{biformula3} p_n(t) & = & p_{n-1}(t) + t^{n-1} p_{n-1}(t) \ene for all $n>1$.  Clearly \eqref{biformula3}, together with the fact that $p_1(t)=1$, uniquely determines the $p_n(t)$.  Since $SO(1)$ is a point, its Poincar\'e polynomial is indeed $p_1(t)=1$, so to show that our formula for the $b_i(n)$ is correct, it is therefore enough to show that the Poincar\'e polynomials of $SO(n)$ satisfy \eqref{biformula3} for all $n>1$.  We assume $n>1$ from now on.

The long exact sequence \eqref{LES} for our Morse function $f$ relates the cohomology of $SO(n)$ to the cohomology of $F_0$ and $F_{n-1}$ (with a degree shift of $n-1$ for the latter), both of which are identified with $SO(n-1)$.  From this sequence, we see that \eqref{biformula3} is equivalent to saying that the maps \be t_* s^* : \H^i(SO(n-1),\FF_2) & \to & \H^{i+1-n}(SO(n-1),\FF_2) \ee are all zero.

Since $n>1$, the relative dimension $n-1$ of $t$ is $>0$, hence it follows from the Projection Formula that $t_* t^* = 0$, so it will be enough to show that $s^* = t^*$ as maps $$\H^i(SO(n-1),\FF_2) \to \H^i(M,\FF_2).$$  To see this, we write $s$ as the composition of \be st^{-1} \times t : M & \to & SO(n-1) \times SO(n-1) \ee and the multiplication map \be \mu : SO(n-1) \times SO(n-1) & \to & SO(n-1). \ee  We are working over a field, so we have a K\"unneth Formula isomorphism \be \H^*(SO(n-1),\FF_2) \otimes \H^*(SO(n-1),\FF_2) & = & \H^*(SO(n-1) \times SO(n-1),\FF_2) \\ \alpha \otimes \beta & \mapsto & (\pi_1^* \alpha )(\pi_2^* \beta). \ee  In terms of this isomorphism, $\mu^*$ can be written \be \mu^* \beta & = & \beta \otimes 1 + 1 \otimes \beta + \sum_j \beta'_j \otimes \beta''_j, \ee where the $\beta'_j$ and $\beta''_j$ have positive degree.\footnote{It is easy to prove that this holds for any $H$-space.  See \cite[Page 283]{H}.  In fact, the Hopf algebra structure on $\H^*(SO(n-1),\FF_2)$ can be shown to be primitive \cite[Page 298]{H}, meaning the terms in the sum over $j$ above are not actually present, but we do not need to make use of this fact.}

Now we get to the (next) \emph{key observation}: According to Proposition~\ref{prop:main}, the map $st^{-1}$ factors through the quotient map $q : S^{n-2} \to \RR \PP^{n-2}$.  But the maps \be q^* : \H^i( \RR \PP^{n-2}, \FF_2 ) & \to & \H^i( S^{n-2}, \FF_2 ) \ee are zero for $i \neq 0$ (this is obvious for $i \neq n-2$ and holds when $i=n-2$ because $q$ has degree $2$), hence $(st^{-1})^* \alpha = 0$ whenever $\alpha$ has positive degree.  The equality that we want to establish, $s^* = t^*$, is obvious in degree $0$, and, for $\beta$ of positive degree, we compute \be s^* \beta & = & ((st^{-1}) \times t)^* \mu^* \beta \\ & = & ((st^{-1})^* \otimes t^* ) ( \beta \otimes 1 + 1 \otimes \beta  + \sum_j \beta'_j \otimes \beta''_j) \\ & = & (st^{-1})^* \beta + t^* \beta + \sum_j (st^{-1})^*(\beta'_j))(t^*(\beta''_j)) \\ & = & t^* \beta. \ee

\section*{Appendix}

Recall from \S\ref{section:SOn1} that Frankel, in \cite{F}, obtained---via Morse Theory---a relationship between the mod $2$ Betti numbers of $SO(n)$ and those of the Grassmannians $\GG(2k,n)$.  Here we will express and prove this relationship in a purely combinatorial manner.  This combinatorial discussion can be interpreted in various ways.  It can be viewed as a derivation of the combinatorial formula for the mod $2$ Betti numbers of $SO(n)$ given in \S\ref{section:SOn2} from Frankel's relationship (assuming, as we will, standard formulae for the mod $2$ Betti numbers of Grassmannians).  Alternatively, it can be viewed as a derivation of Frankel's relationship from the aforementioned combinatorial formula.

It will be convenient to work with $O(n)$ instead of $SO(n)$ to avoid various parity considerations.  For our purposes, Frankel's relationship is most naturally written \bne{FrankelR} \dim \H^i(O(n),\FF_2) & = & \sum_k \dim \H^{i-\iota(k)}(\GG(k,n),\FF_2), \ene where $\iota(k) := \bp k \\ 2 \ep$.  Since $\GG(k,n) \cong \GG(n-k,n)$, the RHS of \eqref{FrankelR} is unchanged if we instead take $\iota(k) := \bp n-k \\ 2 \ep$.  It can be seen similarly that \eqref{FrankelR} is equivalent to \be \dim \H^i(SO(n),\FF_2) &=& \sum_k \dim \H^{i-\iota(2k)}(\GG(2k,n),\FF_2), \ee using the latter definition of the $\iota(k)$.  (This is the form of Frankel's relationship in \S\ref{section:SOn1} and \cite{F}.)

\begin{defn} \label{defn:degree} For a positive integer $n$, set $[n]:= \{ 1,\dots,n \}$.  For $S \subseteq [n]$, define \be \deg S & := & | \{ (p,q) \in ([n] \setminus S) \times S : p < q \} | \\ \sdeg S & := & \bp |S| \\ 2 \ep + \deg S = |\{ (p,q) \in [n]  \times S : p < q\}|.  \ee  If $n$ is not clear from context we will write $\deg_n S$ (resp.\ $\sdeg_n S$) for $\deg S$ (resp.\ $\sdeg S$).  Define \be C_i(k,n) & := & \{ S \subseteq [n] : |S|=k, \; \deg S = i \} \\ C_i(n) & := & \{ S \subseteq [n] : \sdeg S = i \} \\ c_i(k,n) & := & | C_i(k,n) | \\ c_i(n) & := & | C_i(n) | \\ b_i(n) & := & | \{ S \subseteq [n-1] : \sum_{s \in S} s = i \} |. \ee \end{defn}  

The numbers $b_i(n)$ defined above were also defined in \S\ref{section:SOn2}, where it was shown that \be \dim \H^i(SO(n),\FF_2) &=& b_i(n). \ee

It is well-known that \be \dim \H^i(\GG(k,n),\FF_2) & = & c_i(k,n), \ee hence the RHS of \eqref{FrankelR} is nothing but $c_i(n)$.  Our combinatorial version of Frankel's relationship is the formula \bne{combFrankel} 2 b_i(n) & = & c_i(n). \ene  To establish \eqref{combFrankel}, first note that, when $n=1$, both sides of \eqref{combFrankel} are $2$ (resp. $0$) when $i=0$ (resp.\ when $i \neq 0$).  (We have $C_0(1) = \{ \emptyset, \{ 1 \} \}$.)  As in \S\ref{section:SOn2}, we see easily that \bne{brec} b_i(n) &=& b_i(n-1)+b_{i+1-n}(n-1) \ene for all $i$ and all $n>1$.  Together with the known values of the $b_i(1)$, formula \eqref{brec} determines all the $b_i(n)$.  To establish \eqref{combFrankel} it remains only to show that \bne{crec} c_i(n) &=& c_i(n-1)+c_{i+1-n}(n-1) \ene for all $i$ and all $n>1$.  To do this, first observe that, for any $S \subseteq [n-1]$ (which we can also regard as a subset of $[n]$), we have \bne{degrel} \deg_{n-1} S & = & \deg_n S \\ \label{sdegform1} \sdeg_{n-1} S &=& \sdeg_n S. \ene  Next observe that \be & & \{ (p,q) \in ([n] \setminus (S \coprod \{ n \})) \times (S \coprod \{ n \} ) : p < q \} \\ & = & \{ (p,q) \in ([n-1] \setminus S) \times S : p < q \} \coprod \{ (p,n) : p \in ([n-1] \setminus S) \}. \ee  Taking cardinalities yields \bne{nextdegform} \deg_n (S \coprod \{ n \}) &=& n-1-|S|+\deg_{n-1} S. \ene  Adding \eqref{nextdegform} and the formula \be \bp |S|+1 \\ 2 \ep &=& \bp |S| \\ 2 \ep + |S| \ee yields \bne{sdegform2} \sdeg_n( S \coprod \{ n \}) & =& n-1+\sdeg_{n-1} S. \ene  By \eqref{sdegform1} the rule $S \mapsto S$ defines a map $C_i(n-1) \to C_i(n)$ and by \eqref{sdegform2} the rule $S \mapsto S \coprod \{ n \}$ defines a map $C_{i+1-n}(n-1) \to C_i(n)$.  The coproduct of these two maps yields a map \be C_i(n-1) \coprod C_{i+1-n}(n-1) & \to & C_i(n) \ee which is clearly bijective because any $S \in C_i(n)$ either contains $n$ or doesn't.  Taking cardinalities yields the recursion \eqref{crec}.

\end{document}